\documentclass[11pt,reqno]{amsart}

\usepackage{amsmath,amsthm,amssymb,mathtools}
\usepackage{enumitem}
\usepackage{xcolor}
\usepackage{graphicx}
\graphicspath{{figures/}}
\usepackage[colorlinks,citecolor=magenta,linkcolor=blue,urlcolor=black]{hyperref}

\usepackage[T1]{fontenc}
\allowdisplaybreaks

\numberwithin{equation}{section}

\theoremstyle{definition}
\newtheorem{thm}{Theorem}[section]
\newtheorem{cor}[thm]{Corollary}
\newtheorem{lem}[thm]{Lemma}
\newtheorem{prop}[thm]{Proposition}
\newtheorem{defn}[thm]{Definition}
\newtheorem{eg}[thm]{Example}
\newtheorem{rem}[thm]{Remark}
\newtheorem{conj}[thm]{Conjecture}

\newcommand{\Hmod}[1]{\widetilde{\mathsf H}_{#1}}
\newcommand{\hrow}[1]{\widetilde{\mathsf h}_{#1}}
\newcommand{\qq}[1]{(q;q)_{#1}}
\newcommand{\Th}{\Theta}
\newcommand{\LT}{\mathrm{LT}}

\newcommand{\unit}[1]{\mathbf 1_{#1}}
\newcommand{\Pop}{\boldsymbol{\Pi}}
\newcommand{\qbin}[2]{\genfrac[]{0pt}{}{#1}{#2}_q}

\newcommand{\LLT}{\mathcal{G}}
\newcommand{\bfsdeg}{\mathrm{bdeg}}
\newcommand{\OP}{\mathcal{OP}}
\newcommand{\words}[1]{\mathcal{W}(#1)}
\newcommand{\rwords}[2]{\mathcal{W}(#2,#1)}
\newcommand{\cont}{\operatorname{cont}}
\newcommand{\wt}{\operatorname{wt}}

\title[Theta operators at $t=1$ and Macdonald cumulants]
{Theta operators at $t=1$, Macdonald cumulants, and LLT positivity}

\author{Jim Haglund}
\address{Department of Mathematics,
University of Pennsylvania, Philadelphia, PA 19104, USA}
\email{\href{mailto:jim.haglund@upenn.edu}{jim.haglund@upenn.edu}}

\author{Vasu Tewari}
\address{Department of Mathematical and Computational Sciences, University of Toronto Mississauga, Mississauga, ON L5L 1C6, Canada}
\email{\href{mailto:vasu.tewari@utoronto.ca}{vasu.tewari@utoronto.ca}}

\thanks{
VT acknowledges the support of an NSERC Discovery Grant (RGPIN-2024-05433).}

\begin{document}

\begin{abstract}
We study the Theta operators of D'Adderio--Iraci--Vanden Wyngaerd at $t=1$
and show that the power-sum-indexed operators $\Th_{\mathsf p_k}|_{t=1}$ agree with a commuting family of
derivations when restricted to symmetric functions of positive degree.
Writing $\hrow a$ for the modified Macdonald function indexed by the single
row $(a)$ we establish that $\Th_{\mathsf p_\mu}\hrow a|_{t=1}$ is, up to a
normalization, the single-row Macdonald cumulant of Do\l{}\k{e}ga.
We use these results to show that $\Th_{\mathsf p_\mu}\hrow a|_{t=1}$ and
$\Th_{\mathsf e_\lambda}\hrow a|_{t=1}$ are both sums of vertical-strip LLT
polynomials indexed by certain plane trees. The former further shows that
single-row Macdonald cumulants are LLT-positive, thereby yielding a stronger
form of the higher-order Macdonald positivity conjecture of Do\l{}\k{e}ga
when all shapes are single rows.
\end{abstract}

\maketitle


\section{Introduction}\label{sec:intro}

Let $\Lambda$ be the ring of symmetric functions over $\mathbb Q(q,t)$ and
let $\{\Hmod\mu(X;q,t)\}$ be the modified
Macdonald basis \cite[(2.40)]{Hag08} as $\mu$ ranges over all partitions. 
Its Schur coefficients
\[
\Hmod\mu=\sum_\lambda \widetilde K_{\lambda\mu}(q,t)\,\mathsf s_\lambda
\]
are the $q,t$-Kostka polynomials, and \emph{Macdonald positivity} is the
statement that $\widetilde K_{\lambda\mu}(q,t)\in\mathbb Z_{\ge0}[q,t]$.
This was conjectured by Macdonald \cite{Mac88}, \cite[Ch.~VI, \S8]{Mac} and proved
by Haiman \cite{Hai01} by way of the isospectral Hilbert scheme. 
The same basis carries the operator theory of the subject. The nabla operator
$\nabla$ of Bergeron, Garsia, Haiman and Tesler \cite{BGHT} is diagonal on
this basis. Its value $\nabla\mathsf e_n$ is the bigraded Frobenius characteristic of the
diagonal coinvariant ring \cite[Prop.~3.5]{HaimanDH}, and the Shuffle theorem
expresses it as a sum over parking functions \cite{HHLRU,CM}.

Do\l{}\k{e}ga \cite{Dol17} has posed a question strictly more general than
Macdonald positivity. Cumulants of Jack polynomials arose in the work of
Do\l{}\k{e}ga and F\'eray \cite{DF17} on the $b$-conjecture of Goulden and
Jackson \cite{GJ96}. Do\l{}\k{e}ga subsequently carried the construction
to Macdonald polynomials. To each tuple of partitions $\lambda^1,\dots,\lambda^r$
he associates a \emph{Macdonald cumulant}
$\kappa(\lambda^1,\dots,\lambda^r)\in\Lambda$, defined by a normalized
alternating sum of products of modified Macdonald polynomials.
Its Schur coefficients are the \emph{multivariate $q,t$-Kostka polynomials} and lie in
$\mathbb Z[q,t]$ \cite[Thm.~1.5]{Dol17}.

\begin{conj}[{\cite[Conj.~1.6]{Dol17}, \cite[Conj.~1.2]{DolG}}]\label{conj:dolega}
For all partitions $\lambda^1,\dots,\lambda^r$, the Macdonald cumulant
$\kappa(\lambda^1,\dots,\lambda^r)$ is Schur positive, i.e. its Schur
coefficients lie in $\mathbb Z_{\ge0}[q,t]$.
\end{conj}

For $r=1$ Conjecture~\ref{conj:dolega} is exactly Haiman's theorem. Several combinatorial
results support the conjecture. Do\l{}\k{e}ga \cite{DolG} expressed Macdonald
cumulants in terms of $G$-parking functions, obtaining monomial and
quasisymmetric positivity and the Schur coefficients indexed by hooks
\cite[Thm.~6.1]{DolG}. Do\l{}\k{e}ga and Kowalski \cite[Thms.~10--11]{DKLLT}
reduced the conjecture to Schur positivity of their \emph{LLT cumulants},
built from the polynomials of Lascoux, Leclerc and Thibon \cite{LLT}, and
asked whether an LLT cumulant is always a nonnegative combination of 
LLT polynomials.
The answer is yes when every
$\lambda^i=(1)$ \cite{Kow20}, \cite[Thm.~40]{DKLLT}.
Since the $G$-inversion polynomials include the area generating function of parking functions,
Do\l{}\k{e}ga \cite[\S8.2]{DolG} asks for a link between these cumulants
and the Shuffle theorem and Delta conjecture.

Our first result proves a stronger form of Conjecture~\ref{conj:dolega} when every
$\lambda^i$ is a single row.
For arbitrary row lengths we expand these cumulants
explicitly in vertical-strip LLT polynomials with coefficients in
$\mathbb Z_{\ge0}[q]$.
This combinatorial expansion is coarser than
the Schur expansion.
Note that since $\Hmod{(a)}$ does not involve
$t$, neither do these cumulants.

Somewhat unexpectedly, our route goes through another family of operators on
Macdonald polynomials: the \emph{Theta operators} $\Th_f$ of
D'Adderio, Iraci and Vanden Wyngaerd \cite{DIV}.
These were introduced to formulate a compositional refinement of the Delta
conjecture of Haglund, Remmel and Wilson \cite{HRW} in the spirit of \cite{HMZ}.
See \cite{DM,BHMPS} for recent progress in this context.
Our main structural result identifies the power-sum Theta operators at $t=1$
with commuting derivations on symmetric functions of positive degree.
From it we deduce that $\Th_{\mathsf p_\mu}\hrow a|_{t=1}$ is a single-row
Macdonald cumulant up to an explicit $q$-factor for every single row $(a)$
and partition $\mu$; see \eqref{eq:introKarow}.
For single rows this is a link of the kind asked for in \cite[\S8.2]{DolG}.
A second expansion treats the Theta operators indexed by
elementary symmetric functions.  We now describe these results in detail.

\subsection{Theta operators at \texorpdfstring{$t=1$}{t=1}}
For $a\ge0$ write $\hrow a\coloneqq\Hmod{(a)}$ for the modified Macdonald
functions indexed by a single row. These lie in $\Lambda_{\mathbb Q(q)}$ and their products 
$\hrow\lambda\coloneqq\hrow{\lambda_1}\cdots\hrow{\lambda_\ell}$ form a
basis of $\Lambda_{\mathbb Q(q)}$. For $f\in\Lambda$
we define the Theta operator on all of $\Lambda$ by
\[
\Th_f\;\coloneqq\;\Pop\,f[X/M]\,\Pop^{-1},
\qquad M=(1-q)(1-t),
\]
where $\Pop$ is the operator diagonal in the modified Macdonald basis defined in
\S\ref{ssec:symfn}.
Theta operators are linear and multiplicative in their subscript.

In the subject of Macdonald polynomials, positivity in $q$ and $t$ is the recurring difficulty, and
$t=1$ is the tractable specialization. 
For $\nabla$ the reason is immediate:
at $t=1$ both its eigenvectors $\Hmod\mu$ and its eigenvalues factor over the rows of $\mu$ so $\nabla|_{t=1}$
is multiplicative \cite[Thm.~2.1]{BGHT}. Thus studying the  specialized
operator reduces to its values on the $\hrow a$. Nothing of the sort is
available for the Theta operators. They are not diagonal and both the plethysm
$f[X/M]$ and conjugation by $\Pop^{-1}$ can produce poles at $t=1$.
The specialization therefore rests on cancellation of poles and the defining
conjugation obscures the resulting operator.
Our starting point is that a derivation law provides a substitute for multiplicativity.
Since the $\hrow m$ for $m\ge1$ generate $\Lambda_{\mathbb Q(q)}$ freely,
a derivation of this ring is determined by its values on them.

\begin{thm}\label{thm:introDm}
On symmetric functions of positive degree and for each $k\ge1$ we have
\[
\Th_{\mathsf p_k}\big|_{t=1}=D_k,
\qquad\text{where}\qquad
D_k\hrow m=\frac{\hrow{m+k}-\hrow m\hrow k}{q^k-1}\quad(m\ge1)
\]
defines a derivation $D_k$ of $\Lambda_{\mathbb Q(q)}$. The derivations
$D_k$ pairwise commute.
\end{thm}

The case $k=1$ follows from the computation in \cite[Lem.~5.3]{DILRVW}, but the operators for $k\ge2$ and the derivation law are new. 
For $f\in\Lambda_{\mathbb Q(q)}$ Theorem~\ref{thm:introDm} makes
$\Th_f|_{t=1}$ a polynomial in the $D_k$ on positive-degree symmetric functions.

\subsection{Macdonald cumulants at single rows}
Applied to $\hrow a$, composites of the derivations $D_k$ produce cumulants.
Theorem~\ref{thm:Karow} and Proposition~\ref{prop:spec} give for every partition
$\mu=(\mu_1,\dots,\mu_r)$ and $a\ge1$
\begin{equation}\label{eq:introKarow}
\Th_{\mathsf p_\mu}\hrow a\big|_{t=1}
=\frac{\kappa(\mu_1,\dots,\mu_r,a)}{\prod_i (1+q+\cdots+q^{\mu_i-1})}.
\end{equation}
The
cumulant is symmetric in $\mu_1,\dots,\mu_r,a$ while the denominator contains the parts of $\mu$ alone.

We expand \eqref{eq:introKarow} in vertical-strip LLT polynomials $\LLT(T)$ indexed here by plane
rooted trees $T$. These polynomials are Schur positive \cite[Prop.~5.3.1]{HHLRU}. We write $\bfsdeg(T)$
for the sequence recording the number of children of the non-leaf vertices of $T$ in
breadth-first order.

\begin{thm}
\label{thm:mainLLT}
For every partition $\mu=(\mu_1,\dots,\mu_r)$ and every $a\ge1$, we have
\[
\Th_{\mathsf p_\mu}\hrow a\big|_{t=1}
=\sum_{(B_1,\dots,B_s)}
  \ \sum_{\bfsdeg(T)=(a,\mu(B_1),\dots,\mu(B_s))}
  \LLT(T),
\]
where $(B_1,\dots,B_s)$ ranges over the ordered set partitions of $[r]\coloneqq\{1,\dots,r\}$ and
$\mu(B)=\sum_{i\in B}\mu_i$. 
\end{thm}

As a corollary we obtain the following.

\begin{cor}\label{cor:introdolega}
For all positive integers $\nu_1,\dots,\nu_r$ the cumulant $\kappa(\nu_1,\dots,\nu_r)$ is a
$\mathbb Z_{\ge0}[q]$-linear combination of vertical-strip LLT polynomials and hence Schur positive.
\end{cor}

\subsection{The elementary expansion}
\begin{thm}
\label{thm:mainlltrow}
For every $a\ge1$ and every partition $\lambda\vdash n$, we have
\[
\Upsilon^{(a)}_\lambda\coloneqq\Th_{\mathsf e_\lambda}\hrow a\big|_{t=1}
=\sum_{S\in\LT^{(a)}_\lambda}\LLT(S),
\]
a sum over \emph{sibling-increasing labeled $a$-trees}
of \emph{content} $\lambda$.
\end{thm}

Theorem~\ref{thm:mainlltrow} should be compared with two earlier positive
expansions at $t=1$. The work of \cite{DILRVW} treats the special case $a=1$,
with a proved monomial formula for $\lambda=1^n$ \cite[Thm.~4.5]{DILRVW}
and a conjectural one for general $\lambda$.
Theorem~\ref{thm:mainlltrow} does not recover that particular monomial formula.
It would be interesting to reconcile the two expansions combinatorially.
During the preparation of this work D'Adderio, Interdonato, Iraci and Pagaria
\cite{DIIP} obtained explicit formulas for Negu\c{t} operators. Iterating
\cite[Thm.~5.14, Cor.~5.16]{DIIP} writes $\Th_{\mathsf e_\lambda}\mathsf e_1$
as a nonnegative integral combination of values of Negu\c{t} operators at $1$.
Their $\mathsf e$-positive expansions at $q=1$ in \cite[Thm.~4.29]{DIIP}
combine with the $q\leftrightarrow t$ symmetry to give an $\mathsf e$-positive expansion of
$\Upsilon^{(1)}_\lambda$.
Neither of these aforementioned works reaches $a\ge2$ whereas
Theorem~\ref{thm:mainlltrow} treats every $a\ge1$ uniformly.

\subsection{Outline}

Section~\ref{sec:prelim} fixes algebraic and combinatorial conventions.
Section~\ref{sec:seed} identifies
$\Th_{\mathsf p_k}|_{t=1}$ with a derivation $D_k$ and derives
the resulting Hopf-algebraic product rule.
Section~\ref{sec:gamma} introduces an apparently new family $\Gamma_\alpha$ of
symmetric functions indexed by compositions and computes the action of the $D_k$ on it.\footnote{Remark~\ref{rem:Enk} relates this family to the compositional shuffle theorem at $t=1$.}
The tree expansion of $\Gamma_\alpha$ is the subject of Section~\ref{sec:gammatree}.
Section~\ref{sec:expansions} combines
these two facts with the specialization of Section~\ref{sec:seed} to prove the two LLT expansions.
It also gives the expansions for other subscripts.
Section~\ref{sec:rowcum} completes the connection to Macdonald cumulants.

\subsection*{Acknowledgements}
VT is very grateful to Mike Zabrocki for several discussions involving LLT polynomials, which in turn guided the expansions obtained in Sections~\ref{sec:gamma}--\ref{sec:expansions}.
Claude Opus 4.8 and 5 were used to write Sage code for computational verification and experimenting, IPE code for generating figures, and
for language editing. 
\section{Preliminaries and conventions}\label{sec:prelim}

We begin with a quick recall of some standard combinatorial conventions.
We write $[n]=\{1,\dots,n\}$ for $n\ge0$ (so $[0]=\varnothing$) and set
\[
[k]_q = 1+q+\cdots+q^{k-1},\qquad
(q;q)_k=\prod_{i=1}^k(1-q^i),\qquad
\qbin{n}{k}=\frac{(q;q)_n}{(q;q)_k\,(q;q)_{n-k}}
\]
for the $q$-integer, the $q$-Pochhammer symbol (with $(q;q)_0=1$), and the
$q$-binomial coefficient, the last for $0\le k\le n$ and zero otherwise.


\subsection{Compositions and partitions}
A \emph{weak composition} $\alpha=(\alpha_1,\dots,\alpha_m)$ is a finite sequence of nonnegative integers, its \emph{parts}. 
We write $|\alpha|$ for its \emph{size}, the sum $\alpha_1+\cdots+\alpha_m$,
and $\unit i$ for the $i$-th coordinate vector, so that $\alpha+k\unit i$
raises the $i$-th part of $\alpha$ by $k$ and $\alpha-k\unit i$ lowers it by $k$.
If $\alpha$ has no parts equal to $0$, we call it a  \emph{composition}. If $\alpha$ has size $n$, we write it as $\alpha\vDash n$. 
For compositions $\alpha=(\alpha_1,\dots,\alpha_r)$ and
$\beta=(\beta_1,\dots,\beta_s)$ we write
\[
\alpha\cdot\beta=(\alpha_1,\dots,\alpha_r,\beta_1,\dots,\beta_s),
\qquad
\alpha\odot\beta=(\alpha_1,\dots,\alpha_{r-1},\alpha_r+\beta_1,\beta_2,\dots,\beta_s)
\]
for their \emph{concatenation} and \emph{near-concatenation}. These operations govern the
recursion of Section~\ref{sec:gamma}.

A \emph{partition} $\lambda=(\lambda_1,\dots,\lambda_{\ell})$ is a composition satisfying $\lambda_1\ge \cdots \ge\lambda_{\ell}$.
If $\lambda$ has size $n$, we write it as $\lambda\vdash n$.
We further write $m_i(\lambda)$, abbreviated
$m_i$, for the number of parts of $\lambda$ equal to $i$, and let $
\ell(\lambda) = \sum_i m_i$ denote the \emph{length} of $\lambda$. Finally we set $
z_\lambda \coloneqq \prod_i\bigl(i^{m_i}\,m_i!\bigr)$, a quantity routinely encountered in the theory of symmetric functions. 
Two further statistics of $\lambda$ recur throughout, each a product over
its cells $c$, whose arm, leg, coarm and coleg are written
$a(c),l(c),a'(c),l'(c)$:
\begin{gather*}
\Pi_\lambda = \prod_{c\ne(0,0)}\big(1-q^{a'(c)}t^{l'(c)}\big),\qquad
w_\lambda = \prod_{c}\big(q^{a(c)}-t^{l(c)+1}\big)\big(t^{l(c)}-q^{a(c)+1}\big).
\end{gather*}

\subsection{Words}\label{ssec:words}
A \emph{word} is a finite sequence $\mathbf w=(w_1,\dots,w_n)$ of positive
integers, its \emph{letters}. We write
$\mathbf x^{\mathbf w}=x_{w_1}\cdots x_{w_n}$ for its monomial and
$\mathrm{coinv}(\mathbf w)=\#\{i<j:w_i<w_j\}$ for its \emph{coinversion
number}. The \emph{content} of $\mathbf w$ is the weak composition
$\cont(\mathbf w)=(\beta_1,\beta_2,\dots)$ in which $\beta_i$ counts the
letters of $\mathbf w$ equal to $i$, trailing zeros being immaterial; thus
$|\cont(\mathbf w)|=n$.

Now fix a weak composition $\beta$. We write
$\mathbf x^{\beta}=\prod_ix_i^{\beta_i}$, so that
$\mathbf x^{\mathbf w}=\mathbf x^{\beta}$ for every word $\mathbf w$ of
content $\beta$. We let $\words{\beta}$ denote the set of such words, and
for $r\geq 0$ we let $\rwords{r}{\beta}\subseteq\words{\beta}$ be the subset
of words whose first letter $w_1$ strictly exceeds $r$. The generating
functions for coinversions over these sets are
\[
C(\beta)=\sum_{\mathbf w\in\words{\beta}}q^{\mathrm{coinv}(\mathbf w)},
\qquad
C(\beta,r)=\sum_{\mathbf w\in\rwords{r}{\beta}}q^{\mathrm{coinv}(\mathbf w)}.
\]
Note that if $|\beta|=0$ then $\words{\beta}$ consists of the empty word
alone, so $C(\beta)=1$, whereas $\rwords{r}{\beta}=\varnothing$ and
$C(\beta,r)=0$.
These two generating functions are related by the following identity,
which will prove very useful in Section~\ref{sec:gammatree}.

\begin{lem}\label{lem:firstletter}
For every weak composition $\beta$ and every $r\ge0$, we have
\[
\bigl(q^{\sum_{i>r}\beta_i}-1\bigr)\,C(\beta)
=\bigl(q^{|\beta|}-1\bigr)\,C(\beta,r).
\]
\end{lem}

\begin{proof}
Put $b=|\beta|$ and $m=\sum_{i>r}\beta_i$. If $b=0$ both sides vanish, so
let $b\ge1$. Classifying the words in $\rwords{r}{\beta}$ by their first
letter, which forms a coinversion with each later letter exceeding it, gives
\[
C(\beta,r)=\sum_{i>r,\ \beta_i\ge1}q^{\sum_{j>i}\beta_j}\,C(\beta-\unit i).
\]
By MacMahon's theorem
\cite[Prop.~1.7.1]{StanEC1} we have that $C(\beta)$ is the $q$-multinomial coefficient
$\qq b/\prod_i\qq{\beta_i}$, so that
$(1-q^b)\,C(\beta-\unit i)=(1-q^{\beta_i})\,C(\beta)$ for $\beta_i\ge1$.
Hence
\[
(1-q^b)\,C(\beta,r)
=C(\beta)\sum_{i>r}\bigl(q^{\sum_{j>i}\beta_j}-q^{\sum_{j\ge i}\beta_j}\bigr)
=(1-q^{m})\,C(\beta).\qedhere
\]
\end{proof}


\subsection{Plane trees}\label{ssec:trees}

A \emph{plane rooted tree} is a finite rooted tree in which the children of
each vertex are linearly ordered. The number of children of a vertex is its
\emph{out-degree}, and non-root vertices are \emph{active}. The \emph{depth}
of a vertex is the number of edges on the unique path from the root to it, so
that the root has depth $0$. Throughout this section and
Sections~\ref{sec:gamma}--\ref{sec:expansions} the root is inert: it carries no data
and enters every statement below only through its out-degree. 

\begin{defn}\label{def:atree}
Let $a\ge1$. An \emph{$a$-tree} is a plane rooted tree whose root has exactly $a$ children. Equivalently it is an ordered $a$-tuple
of plane rooted trees, its \emph{components}.
\end{defn}

We use \emph{depth-first search} (DFS) and \emph{breadth-first search} (BFS).
DFS visits the child subtrees from left to right and completes each subtree
before moving to the next. BFS visits vertices in increasing order of depth
and from left to right at each depth.

The \emph{BFS degree word} $\bfsdeg(T)$ records the
out-degrees of the non-leaf vertices in BFS order. The root is a non-leaf, so
its out-degree is the first entry.
Suppose now that $T$ has $n$ active vertices, listed in DFS order as
$v_1,\dots,v_n$, the root not being among them, and set
$\mathrm{depth}_i\coloneqq \mathrm{depth}(v_i)$.
The resulting sequence $(\mathrm{depth}_1,\dots,\mathrm{depth}_n)$ satisfies
$\mathrm{depth}_{i+1}\le\mathrm{depth}_i+1$, with equality exactly when
$v_{i+1}$ is the leftmost child of $v_i$. The maximal segments of this sequence along
which the depth rises by exactly one at each step are the \emph{vertical runs}
of $T$.
Figure~\ref{fig:peel} shows a tree $T$ whose labels give the DFS
order: DFS visits
$\mathrm{root},v_1,v_2,v_3,v_4,v_5,v_6$, whereas BFS visits
$\mathrm{root},v_1,v_5,v_2,v_4,v_6,v_3$. Consequently,
$\bfsdeg(T)=(2,2,1,1)$, and the vertical runs are
$(v_1,v_2,v_3)$, $(v_4)$, and $(v_5,v_6)$.

\begin{figure}[ht]
\centering
\includegraphics{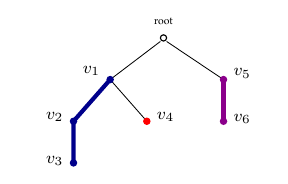}
\caption{A plane rooted tree with active vertices in DFS order and vertical runs in bold.}
\label{fig:peel}
\end{figure}

\subsection{Symmetric functions and Macdonald polynomials}\label{ssec:symfn}
Recall from the introduction that  $\Lambda$ is the ring of symmetric functions over $\mathbb Q(q,t)$.
A subscript on \(\Lambda\) always denotes a coefficient ring and is omitted when it is clear.
We write \(\Lambda^{(n)}\) for the homogeneous component of degree \(n\) and
\(\Lambda^+=\bigoplus_{n\ge1}\Lambda^{(n)}\) for the positive-degree part.
Our
results lie in \(\Lambda_{\mathbb Q(q)}\), and unless the variable \(t\) is
displayed, every identity is understood in this ring. Intermediate quantities
are rational in \(t\) and do have poles at \(t=1\), so \(t=1\) is substituted
only in combinations shown to be regular there.
An \emph{alphabet} $A$ is a formal sum of monomials in the ambient variables
(which include $q$ and $t$).
Symmetric functions in an auxiliary alphabet
\(A\) are written \(\Lambda(A)\), with the default alphabet always equal to $X=x_1+x_2+\cdots$.

Given an alphabet $A$, plethystic substitution is the $\mathbb
Q$-algebra map $f\mapsto f[A]$ determined by the rule that
$\mathsf p_r[A]$ is the sum of the $r$-th powers of the monomials of $A$, taken
with multiplicity. Reading $
M=1-q-t+qt$ 
as an alphabet gives $
\mathsf p_r[MX] = (1-q^r)(1-t^r)\,\mathsf p_r[X]$,
so that
\[
f^\ast\coloneqq f[X/M]
\qquad\text{is the substitution}\qquad
\mathsf p_r\longmapsto\frac{\mathsf p_r}{(1-q^r)(1-t^r)}.
\]
The \emph{Hall inner product} on $\Lambda$ is defined by declaring $
\langle \mathsf p_\lambda,\mathsf p_\mu\rangle = \delta_{\lambda\mu}\,z_\lambda$.
The involution $\omega$ is given by $\omega\,\mathsf p_\lambda=\varepsilon_\lambda\,\mathsf p_\lambda$
where $
\varepsilon_\lambda\;\coloneqq\;(-1)^{|\lambda|-\ell(\lambda)}$.
Define the \emph{star pairing} by
\[
\langle f,g\rangle_* \coloneqq \langle f,\ \omega\, g[MX]\rangle.
\]
This pairing is diagonal on power sums.

The modified Macdonald basis is orthogonal for this pairing \cite{GHT}:
$\langle\Hmod\lambda,\Hmod\mu\rangle_* = \delta_{\lambda\mu}\,w_\mu$.
For every homogeneous $f\in \Lambda$ of degree $n$ this gives
\begin{equation}\label{eq:starorth}
f = \sum_{\mu\vdash n}\langle f,\Hmod\mu\rangle_*\,\frac{\Hmod\mu}{w_\mu}.
\end{equation}

\begin{defn}\label{def:theta_again}
Let $\Pop$ denote the linear operator on $\Lambda$ determined by  $
    \Pop\, \Hmod{\lambda}=\Pi_{\lambda}\Hmod{\lambda}.
$
For $f\in\Lambda$ the Theta operator is
$\Th_f=\Pop\,f^\ast\,\Pop^{-1}$. We call $f$ the \emph{subscript} of $\Th_f$.
\end{defn}
 Since
$f^\ast g^\ast=(fg)^\ast$, we have that $f\mapsto\Th_f$ is linear and
\emph{multiplicative}: $
\Th_{fg}=\Th_f\Th_g$ for $f,g\in\Lambda
$.

\begin{rem}\label{rem:unitconv}
The published definition \cite[(28)]{DIV} and the references \cite{DR,IR}
set $\Th_fg=0$ when $\deg g=0<\deg f$.
For $n\ge1$ our convention gives $
\Th_{\mathsf e_n}(1)=\Pop\,\mathsf e_n^\ast
=\frac{(-1)^{n-1}\mathsf p_n}{M\,[n]_q\,[n]_t}$,
whereas it is $0$ in theirs. The two conventions agree on $\Lambda^+$,
which is where our specialization results apply. In our convention these results do not extend to constants since $\Th_{\mathsf e_n}(1)$ has a simple pole at $t=1$. Theorem~\ref{thm:Dm} holds on all of $\Lambda$ in the convention of \cite{DIV} because both sides vanish on constants.
\end{rem}

\subsection{The vertical-strip LLT polynomial}\label{ssec:llt}
We introduce vertical-strip LLT polynomials using plane rooted trees,
and relate this indexing to tuples of vertical strips in
Remark~\ref{rem:normalization}.

\begin{defn}\label{def:coloring}
Let $v_1,\dots,v_n$ be the active vertices of $T$ in DFS order. A
\emph{coloring} of $T$ is a word $\mathbf c=(c_1,\dots,c_n)$ of
positive integers, one color $c_i$ for each active vertex $v_i$. It is
\emph{valid} if it satisfies the \emph{strict-rise condition}:
$
\mathrm{depth}_{i+1}=\mathrm{depth}_i+1
\quad\Longrightarrow\quad
c_i<c_{i+1}
\text{ for } 1\le i<n.$
We write $\mathrm{Col}(T)$ for the set of valid colorings of $T$.
\end{defn}
Equivalently, $\mathbf c\in \mathrm{Col}(T)$ when its colors strictly increase along
every vertical run of $T$. The root is not colored.
We are now ready to introduce the vertical-strip
LLT polynomial \cite{LLT}.

\begin{defn}\label{def:llt}
Let $\mathbf c$ be a coloring of $T$. For $1\leq i<j\leq n$, the pair
$(i,j)$ is \emph{primary} if
$\mathrm{depth}_i=\mathrm{depth}_j$ and $c_i<c_j$, and \emph{secondary} if
$\mathrm{depth}_i=\mathrm{depth}_j+1$ and $c_i>c_j$. 
The \emph{diagonal
inversion number} $\mathrm{dinv}(\mathbf c)$ is the number of pairs of either
kind, and the \emph{weight} of a valid coloring $\mathbf c$ is
$\wt(\mathbf c)\coloneqq q^{\mathrm{dinv}(\mathbf c)}\mathbf x^{\mathbf c}$.
The \emph{LLT polynomial} of $T$ is the sum of the weights of its valid
colorings:
\[
\LLT(T)\;\coloneqq\;\sum_{\mathbf c\in \mathrm{Col}(T)}\wt(\mathbf c).
\]
\end{defn}

This distinguished family is Schur positive
\cite[Prop.~5.3.1]{HHLRU}. This rests on the identification of LLT products of
partition shapes with parabolic Kazhdan--Lusztig polynomials by Leclerc and
Thibon \cite{LT}, and on the positivity of the latter due to Kashiwara and
Tanisaki \cite{KT}.

Figure~\ref{fig:dinv} shows a valid coloring of a tree whose vertical runs in DFS order are
\[
R_1=(v_1,v_2,v_3),\qquad R_2=(v_4),\qquad
R_3=(v_5,v_6,v_7),\qquad R_4=(v_8).
\]
The numbers in the circles are the colors. Thus
$\mathbf c=(2,3,4,1,1,3,4,3)$ and
$\mathbf x^{\mathbf c}=x_1^2x_2x_3^3x_4^2$.
The right-hand diagram displays the associated tuple of vertical strips.

\begin{figure}[ht]
\centering
\includegraphics{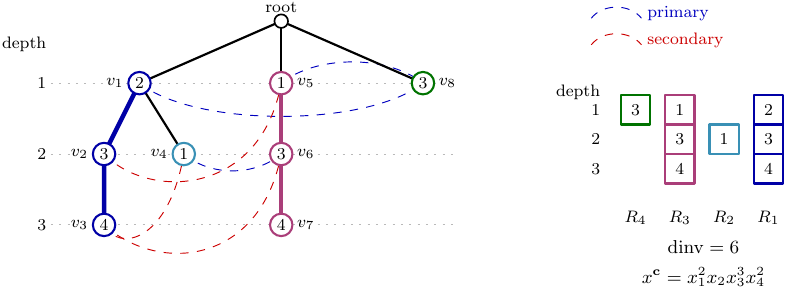}
\caption{A valid coloring of a second tree with its primary and secondary pairs.
}
\label{fig:dinv}
\end{figure}

\begin{rem}\label{rem:normalization}
Encode the DFS tour by a word in $\{U,D\}$, writing $U$ for an edge traversed
away from the root and $D$ for an edge traversed toward it. The resulting
word is a Dyck word, and its maximal $U$-runs are precisely the vertical runs
of $T$. Taking those runs in reverse order and giving each cell its
depth as content, the primary and secondary pairs of Definition~\ref{def:llt}
are exactly the Bylund--Haiman inversion pairs of the resulting tuple of
vertical strips. Thus Definition~\ref{def:llt} is the normalization of
\cite[(6.8) and Rem.~6.5]{Hag08}, and Figure~\ref{fig:dinv} records one
coloring both ways. Carlsson and Mellit \cite{CM} order $\mathbb Z_{>0}$
oppositely, so their columns decrease and their
$\mathrm{dinv}$ inequalities are reversed. Let $\pi$ be the DFS Dyck path
of $T$ and let $\pi'$ be its image under the path bijection in
\cite[\S2.4]{CM}. Then $\LLT(T)=\chi(\pi',0)$ by \cite[Rem.~3.6]{CM},
while their $\chi(\pi')$ of \cite[Def.~3.1]{CM} is the corresponding unicellular
polynomial. Kowalski \cite[\S2.2]{Kow20} employs the same passage from a tree
to a tuple of vertical strips.
\end{rem}



\subsection{Auxiliary results}
We now record a few results that shall prove helpful in the sequel, beginning with an adjointness identity.
Write $f^\perp$ for the Hall
adjoint of multiplication by $f$.
We then have (cf. \cite[(1.32)]{DR})
\[
\langle \mathsf p_k^\ast f,\,g\rangle_\ast
=(-1)^{k-1}\langle f,\,\mathsf p_k^\perp g\rangle_\ast
\qquad(f,g\in\Lambda).
\]
The operator $\Pop$ is diagonal in the star-orthogonal Macdonald basis
and hence star-self-adjoint, so for \(g\in\Lambda^+\) and every partition
\(\nu\) we obtain
\begin{equation}\label{eq:adjcoef}
[\Hmod\nu]\,\Th_{\mathsf p_k}g
=(-1)^{k-1}\,\frac{\Pi_\nu}{w_\nu}\,
\big\langle \Pop^{-1}g,\ \mathsf p_k^\perp\Hmod\nu\big\rangle_\ast .
\end{equation}

The convention of \S\ref{ssec:symfn} permits the substitution $t=1$ only in
combinations that are regular there. We now define that notion.
Since $\Lambda_{\mathbb Q(q,t)}
=\Lambda_{\mathbb Q(q)}\otimes_{\mathbb Q(q)}\mathbb Q(q,t)$, every basis of
$\Lambda_{\mathbb Q(q)}$ is a basis of $\Lambda_{\mathbb Q(q,t)}$ over
$\mathbb Q(q,t)$. Call $f\in\Lambda_{\mathbb Q(q,t)}$ \emph{regular at $t=1$}
if its coefficients in one such basis have no pole at
$t=1$ when regarded
as rational functions of $t$ over $\mathbb Q(q)$. Evaluating these
coefficients at $t=1$ defines the specialization
$f|_{t=1}\in\Lambda_{\mathbb Q(q)}$. Neither the regularity nor the specialization depends on the basis since any two such bases are related by an invertible matrix over $\mathbb Q(q)$.

The modified Macdonald polynomials are regular at $t=1$: their coefficients
in the Schur basis lie in $\mathbb Q[q,t]$, this being the polynomiality of
the $q,t$-Kostka coefficients (Garsia--Tesler \cite{GT}, Knop \cite{Knop},
Sahi \cite{Sahi}, Kirillov--Noumi \cite{KN}). 
We may therefore specialize $\Hmod\mu$ at $t=1$, and doing so produces a
basis of $\Lambda_{\mathbb Q(q)}$. Indeed,
\begin{equation}\label{eq:factort1}
\Hmod\mu(X;q,1) \;=\; \prod_i\hrow{\mu_i}(X;q),
\qquad\text{where}\qquad
\hrow k(X;q)\;\coloneqq\;\qq{k}\,\mathsf h_{k}\big[X/(1-q)\big];
\end{equation}
see \cite[(2.4)]{BGHT}, attributed there to
\cite[Ch.~VI, p.~364, Ex.~7]{Mac}. The products
$
\hrow\gamma \;\coloneqq\; \prod_i\hrow{\gamma_i}
$
over partitions form a basis of $\Lambda_{\mathbb Q(q)}$, the \emph{row
basis}. 
We use the same notation for an arbitrary finite sequence $\gamma$ of
nonnegative integers, and note that $\hrow0=1$.

Being multiplicative, the row basis needs a combinatorial description only
for the $\hrow k$.  One has
\begin{equation}\label{eq:rowcoinv}
\hrow k \;=\; \sum_{\mathbf w}q^{\mathrm{coinv}(\mathbf w)}\mathbf x^{\mathbf w},
\end{equation}
the sum over all words $\mathbf w$ of length $k$.
It is \eqref{eq:rowcoinv} that
drives the tree-indexed LLT expansions of
Sections~\ref{sec:gammatree} and~\ref{sec:expansions}.

These facts are used below without further reference. We shall also need
the converse direction of \eqref{eq:factort1}, since our arguments expand
elements of the row basis in the Macdonald basis and only then let $t\to1$.

\begin{lem}\label{lem:row-transition}
For $n\ge1$, the transition matrices between $\{\Hmod\mu\}_{\mu\vdash n}$ and
$\{\hrow\mu\}_{\mu\vdash n}$ are regular at $t=1$, and both specialize there
to the identity matrix.
\end{lem}
\begin{proof}
The Schur expansion of $\{\Hmod\mu\}$ specializes at $t=1$ to that of
$\{\hrow\mu\}$, which is invertible with entries in $\mathbb Q(q)$. The matrix expressing
$\{\Hmod\mu\}$ in $\{\hrow\mu\}$ is therefore regular at $t=1$ with value
the identity, and hence so is its inverse.
\end{proof}

The $\Hmod\mu$ do not lie in $\Lambda_{\mathbb Q(q)}$, so expansions in them
are governed by Lemma~\ref{lem:row-transition}. The following two consequences of the
definition are used throughout.
\begin{enumerate}[label=\textup{(R\arabic*)},ref=\textup{(R\arabic*)}]
\item\label{it:Rring} Sums and products of elements regular at $t=1$, and of
matrices whose entries are regular at $t=1$, are again regular, and
specializing at $t=1$ commutes with both operations.
\item\label{it:Rrow} The matrix of a linear map between homogeneous
components is regular at $t=1$ in the Macdonald bases if and only if it is
regular in the row bases, and the two specializations agree. This is
Lemma~\ref{lem:row-transition} together with \ref{it:Rring}, the transition
matrices being regular with value the identity.
\end{enumerate}

\section{Specializing the Theta operator at \texorpdfstring{$t=1$}{t=1}}\label{sec:seed}


The dual Cauchy identity lets us record the families
$\{\Th_{\mathsf e_\lambda}g\}$ and $\{\Th_{\mathsf p_\mu}g\}$ with $g\in\Lambda^+$
in a single generating series over an auxiliary alphabet $Y$.
Two commuting alphabets $X,Y$ give $\Lambda(X)\otimes\Lambda(Y)$ and we write
$\widehat{\Lambda(X)\otimes\Lambda(Y)}$ for its completion by $Y$-degree.
In this completion set $\mathsf E[XY]\coloneqq\sum_{k\ge0}\mathsf e_k[XY]$.
The first equality below is the dual Cauchy identity and the second
expands $\mathsf E[XY]$ in power sums:
\begin{equation}\label{eq:EXY}
\mathsf E[XY]
\;=\; \sum_\lambda \mathsf e_\lambda(X)\,\mathsf m_\lambda(Y)
\;=\;\sum_{\mu}\varepsilon_\mu z_\mu^{-1}\,\mathsf p_\mu(X)\,\mathsf p_\mu(Y).
\end{equation}
Here $\mathsf m_\lambda$ is the monomial symmetric function and both sums
run over all partitions.
The operator $\Th$ acts only through the $X$-factor.

\begin{prop}\label{prop:master}
For every $g\in\Lambda^+$ we have in
$\widehat{\Lambda(X)\otimes\Lambda(Y)}$
\[
\sum_{\lambda}\mathsf m_\lambda(Y)\;\Th_{\mathsf e_\lambda}g
\;=\;\Th_{\mathsf E[XY]}\,g
\;=\;\sum_{\mu}\varepsilon_\mu z_\mu^{-1}\,\mathsf p_\mu(Y)\;\Th_{\mathsf p_\mu}g.
\]
\end{prop}

\begin{proof}
Both equalities are the linearity of $\Th$ in its
subscript applied to \eqref{eq:EXY}.
\end{proof}



The $t\to1$ analysis needs the behaviour of a single singular quantity,
$\Pi_\mu/w_\mu$.

\begin{lem}\label{lem:leadcoef}
Fix $\mu\vdash n$ with $n\ge1$ and $\ell=\ell(\mu)$. Regarded as a rational
function of $t$ over $\mathbb Q(q)$, the ratio $w_\mu/\Pi_\mu$ has a simple
zero at $t=1$: the limit
\[
W_\mu\;\coloneqq\;\lim_{t\to1}\frac{(1-t)\,\Pi_\mu}{w_\mu}
\;=\;\frac{\varepsilon_\mu\,(\ell-1)!}
{\prod_i m_i(\mu)!\ \prod_j\qq{\mu_j}}
\]
exists and is a nonzero rational function of $q$.
\end{lem}

\begin{proof}
Every factor of $\Pi_\mu$ and of $w_\mu$ that vanishes at $t=1$ does so
simply, since it has the form $1-t^{d}=(1-t)[d]_t$ and
$[d]_t|_{t=1}=d$. It therefore suffices to locate the vanishing factors,
record their normalized leading coefficients, and evaluate the remaining
factors at $t=1$.

In $\Pi_\mu$ the factor $1-q^{a'(c)}t^{l'(c)}$ vanishes at $t=1$ exactly when
$a'(c)=0$, that is, when $c$ lies in the first column. Excluding $(0,0)$, these
are the cells of coleg $1,\dots,\ell-1$, so there are $\ell-1$ of them and
their leading coefficients multiply to $(\ell-1)!$. Each remaining cell has $a'(c)\ge1$
and value $1-q^{a'(c)}$, independent of $l'(c)$. The coarms in row $j$ are
$1,\dots,\mu_j-1$, so these cells contribute $\prod_j\qq{\mu_j-1}$.

In $w_\mu$ the factor $q^{a(c)}-t^{l(c)+1}$ vanishes at $t=1$ exactly when
$a(c)=0$, while $t^{l(c)}-q^{a(c)+1}$ never does.
There is one arm-$0$ cell per row and hence $\ell$ vanishing factors.
Within a block of $b$ equal parts their legs run $0,\dots,b-1$,
so the normalized leading coefficients contribute $b!$ and give
$\prod_i m_i(\mu)!$ over all blocks.
Evaluating the remaining factors at $t=1$ row by row, we find that the arm-$0$
cell of row $j$ leaves $1-q$ and its arm-$\ge1$ cells leave
\[
\prod_{a=1}^{\mu_j-1}(q^{a}-1)(1-q^{a+1})
=(-1)^{\mu_j-1}\,\frac{\qq{\mu_j-1}\,\qq{\mu_j}}{1-q},
\]
so row $j$ contributes $(-1)^{\mu_j-1}\qq{\mu_j-1}\qq{\mu_j}$ and all rows
together $\varepsilon_\mu\prod_j\bigl(\qq{\mu_j-1}\qq{\mu_j}\bigr)$.

Thus $\Pi_\mu$ vanishes to order $\ell-1$ and $w_\mu$ to order $\ell$, whence
the simple zero of $w_\mu/\Pi_\mu$. In $(1-t)\Pi_\mu/w_\mu$ the
powers of $1-t$ and the factors $\prod_j\qq{\mu_j-1}$ cancel, leaving
the asserted value of $W_\mu$.
\end{proof}

The identification of the first-column and arm-$0$ cells as the factors
vanishing at $t=1$ also appears in the proof of \cite[Lem.~4.2]{IR}.

\begin{rem}
The formula for $W_\mu$ in Lemma~\ref{lem:leadcoef} is symmetric in the
parts of $\mu$. We may therefore read it as the definition of
$W_\gamma$ for an arbitrary finite sequence $\gamma$ of positive integers, so
that $W_\gamma=W_{\gamma^+}$ with $\gamma^+$ the weakly decreasing
rearrangement of $\gamma$. This matches the convention adopted for
the $\hrow\gamma$.
\end{rem}

Set $\lambda_{\ell+1}=0$ for $\ell=\ell(\lambda)$. Then for $1\le i\le\ell+1$ the sequence
$\lambda+k\unit i$ runs over the two moves that enter below: raising a part
$r=\lambda_i\ge1$ to $r+k$, and, for $i=\ell+1$, appending a new part $k$.
The single ratio governing both is the following.

\begin{cor}\label{cor:Wratio}
Let $\lambda\vdash n\ge1$, let $k\ge1$, and let $\nu=\lambda+k\unit i$ with
$1\le i\le\ell(\lambda)+1$ and $r=\lambda_i$. Then we have
\[
\frac{W_\nu}{W_\lambda}
=(-1)^k\,\frac{\widehat m_r(\lambda)}{m_{r+k}(\nu)}\cdot\frac{\qq r}{\qq{r+k}},
\qquad\text{where}\quad
\widehat m_r(\lambda)\coloneqq
\begin{cases}
m_r(\lambda), & r\ge1,\\
-\ell(\lambda), & r=0.
\end{cases}
\]
\end{cor}
\begin{proof}
Immediate from Lemma~\ref{lem:leadcoef}. The case $r=0$ is the only one that is not
literal bookkeeping. There $\ell(\nu)=\ell(\lambda)+1$ so
$(\ell(\nu)-1)!/(\ell(\lambda)-1)!=\ell(\lambda)$ and
$\varepsilon_\nu=(-1)^{k-1}\varepsilon_\lambda$ rather than $(-1)^{k}\varepsilon_\lambda$.
The value $\widehat m_0(\lambda)=-\ell(\lambda)$ accounts for both.
\end{proof}

\subsection{The derivation operators}\label{ssec:derivation}

By Lemma~\ref{lem:leadcoef} the pole of $\Pi_\nu/w_\nu$ cancels the zero
of $w_\mu/\Pi_\mu$ in \eqref{eq:adjcoef},
giving the following coefficient formula in the row basis.

\begin{lem}\label{lem:residue-transfer}
Let $n,k\ge1$, and let $\lambda\vdash n$ and $\nu\vdash n+k$.  Then $\Th_{\mathsf p_k}\hrow\lambda$ is regular at $t=1$ and we have
\[
[\hrow\nu]\,\Th_{\mathsf p_k}\hrow\lambda\big|_{t=1}
=(-1)^{k-1}\frac{W_\nu}{W_\lambda}
\big([\hrow\lambda]\,\mathsf p_k^\perp\hrow\nu\big).
\]
\end{lem}

\begin{proof}
For $\mu\vdash n$ and $\nu\vdash n+k$ put
$d_{\mu\nu}=[\Hmod\mu]\,\mathsf p_k^\perp\Hmod\nu$.
For $g=\Hmod\mu$ we have $\Pop^{-1}g=\Pi_\mu^{-1}\Hmod\mu$,
so evaluating the star pairing in \eqref{eq:adjcoef} by \eqref{eq:starorth} gives
\[
[\Hmod\nu]\,\Th_{\mathsf p_k}\Hmod\mu
=(-1)^{k-1}
\Big(\frac{\Pi_\nu}{w_\nu}\cdot\frac{w_\mu}{\Pi_\mu}\Big)d_{\mu\nu},
\]
an identity of rational functions of $t$ over $\mathbb Q(q)$.
The matrix $(d_{\mu\nu})$ represents $\mathsf p_k^\perp$ in the Macdonald bases.
Since the row-basis matrix of $\mathsf p_k^\perp$ has entries in $\mathbb Q(q)$,
\ref{it:Rrow} shows that $(d_{\mu\nu})$ is regular at $t=1$
with specialization $\big([\hrow\mu]\,\mathsf p_k^\perp\hrow\nu\big)$.

The scalar factor can be written as
$\frac{(1-t)\Pi_\nu}{w_\nu}\big(\frac{(1-t)\Pi_\mu}{w_\mu}\big)^{-1}$.
Since $W_\mu\ne0$, Lemma~\ref{lem:leadcoef} shows that it is regular at $t=1$
with value $W_\nu/W_\mu$.
Hence the Macdonald-basis matrix of
$\Th_{\mathsf p_k}\colon\Lambda^{(n)}\to\Lambda^{(n+k)}$ is regular at $t=1$
with specialized entry
$(-1)^{k-1}(W_\nu/W_\mu)\,[\hrow\mu]\,\mathsf p_k^\perp\hrow\nu$
in position $(\nu,\mu)$.

The row-basis matrix has the same regularity and specialization by \ref{it:Rrow}.
Its entries are the coefficients $[\hrow\nu]\,\Th_{\mathsf p_k}\hrow\lambda$,
which proves both assertions.
\end{proof}

The coefficient on the right in Lemma~\ref{lem:residue-transfer} is
elementary as the next lemma shows.

\begin{lem}\label{lem:pperp}
Let $k\ge1$. For every finite sequence $\nu$ of nonnegative integers we have
\[
\mathsf p_k^\perp\hrow\nu
=\qq{k-1}\sum_{j:\ \nu_j\ge k}\qbin{\nu_j}{k}\,\hrow{\nu-k\unit j}.
\]
\end{lem}
\begin{proof}
Since $\mathsf p_k^\perp=k\,\partial/\partial\mathsf p_k$
\cite[Ch.~I, \S5]{Mac} and $\mathsf p_m[X/(1-q)]=\mathsf p_m/(1-q^m)$, we
have $\mathsf p_k^\perp\bigl(f[X/(1-q)]\bigr)
=(1-q^k)^{-1}(\mathsf p_k^\perp f)[X/(1-q)]$ for all $f$. 
Using $\mathsf p_k^\perp\mathsf h_n=\mathsf h_{n-k}$
and \eqref{eq:factort1} gives
\begin{equation}\label{eq:pperprow}
\mathsf p_k^\perp\hrow n
=\frac{\qq n}{(1-q^k)\,\qq{n-k}}\,\hrow{n-k}
=\qq{k-1}\qbin nk\,\hrow{n-k}\qquad(n\ge k),
\end{equation}
and $\mathsf p_k^\perp\hrow n=0$ for $n<k$. The lemma now follows by the Leibniz
rule as $\mathsf p_k^\perp$ is a derivation.
\end{proof}

Thus $[\hrow\lambda]\,\mathsf p_k^\perp\hrow\nu$ vanishes unless $\nu$ is
obtained from $\lambda$ by raising one part by $k$ or by appending a part
$k$: the right side in Lemma~\ref{lem:residue-transfer} selects exactly these two
moves on the row index. The following operator records them with the
coefficients that arise.

\begin{defn}\label{def:Dk}
For $k\ge1$ let $D_k$ be the linear operator on $\Lambda_{\mathbb Q(q)}$
determined by
\[
D_k\big(\hrow\lambda\big)\;=\;\frac{1}{q^k-1}
\Big(\sum_{i=1}^{\ell(\lambda)}\hrow{\lambda+k\unit i}\;-\;\ell(\lambda)\,
\hrow{\lambda\cup(k)}\Big).
\]
\end{defn}

\begin{lem}\label{lem:Dkder}
Each $D_k$ is a derivation of $\Lambda_{\mathbb Q(q)}$,
\[
D_k(fg)=D_k(f)\,g+f\,D_k(g),
\qquad D_k(1)=0 ,
\]
and the $D_k$ pairwise commute: $[D_j,D_k]=0$ for all $j,k\ge1$.
\end{lem}
\begin{proof}
The rule of Definition~\ref{def:Dk} is additive under
$\lambda\mapsto\lambda\cup\rho$, which is the derivation law on the
multiplicative basis $\{\hrow\lambda\}$, and $D_k(1)=0$ because the empty
partition has length $0$. Since $[D_j,D_k]$ is again a derivation, it
vanishes as soon as it vanishes on the algebra generators $\hrow n$.
A direct calculation gives the following identity,
whose right-hand side is symmetric in $j$ and $k$.
\begin{equation*}
(q^j-1)(q^k-1)\,D_jD_k\hrow n
=\hrow{n+j+k}-\hrow{(n+k,j)}-\hrow{(n+j,k)}-\hrow{(n,j+k)}
+2\,\hrow{(n,j,k)}.\qedhere
\end{equation*}
\end{proof}

\begin{thm}\label{thm:Dm}
As operators on the positive-degree part
$\Lambda^+_{\mathbb Q(q)}$ we have
\[
\Th_{\mathsf p_k}\big|_{t=1}=D_k
\qquad(k\ge1).
\]
\end{thm}
\begin{proof}
Fix $k\ge1$. By Lemma~\ref{lem:residue-transfer}, $\Th_{\mathsf p_k}\hrow\lambda$
is regular at $t=1$ for every partition $\lambda$ of positive size, so
specializing on the row basis defines the linear operator
$\Th_{\mathsf p_k}|_{t=1}$ on $\Lambda^+_{\mathbb Q(q)}$. Both it and $D_k$
raise degree by $k$, so it suffices to compare
$[\hrow\nu]\,\Th_{\mathsf p_k}\hrow\lambda|_{t=1}$ with
$[\hrow\nu]\,D_k\hrow\lambda$ for $\lambda\vdash n\ge1$ and $\nu\vdash n+k$.
As in Corollary~\ref{cor:Wratio} set $\ell=\ell(\lambda)$ and $\lambda_{\ell+1}=0$.
Say that $\nu$ is \emph{reachable} if $\nu=\lambda+k\unit i$ as multisets
for some $1\le i\le\ell+1$. In that case put $r=\lambda_i$. The argument below also shows that $r$ is determined by $\nu$.
The move raises a part $r$ to $r+k$ when $r\ge1$,
while for $r=0$ it appends a part $k$ and gives $\nu=\lambda\cup(k)$.
Definition~\ref{def:Dk} says that
$[\hrow\nu]\,D_k\hrow\lambda$ vanishes for $\nu$ not reachable and equals
$\widehat m_r(\lambda)/(q^k-1)$ otherwise.
For $r\ge1$ the $m_r(\lambda)$ indices $i$ with $\lambda_i=r$
all yield the same $\hrow\nu$, while for $r=0$
the coefficient of $\hrow{\lambda\cup(k)}$ is $-\ell(\lambda)$.

By Lemmas~\ref{lem:residue-transfer} and~\ref{lem:pperp},
\begin{equation}\label{eq:Dmgoal}
[\hrow\nu]\,\Th_{\mathsf p_k}\hrow\lambda\big|_{t=1}
=(-1)^{k-1}\,\frac{W_\nu}{W_\lambda}\,\qq{k-1}\sum_{j}\qbin{\nu_j}{k},
\end{equation}
the sum over those $j$ with $\nu_j\ge k$ for which $\nu-k\unit j$ equals
$\lambda$ as a multiset, zero parts being discarded. This sum is empty unless
$\nu$ is reachable, which settles the vanishing case. If $\nu$ is reachable,
the contributing $j$ are exactly those with $\nu_j=r+k$: removing $k$ from a
part $r+k$ of $\nu$ returns $\lambda$, and a part $s+k$ with $s\ne r$ does
not, since $\{r+k,s\}\ne\{r,s+k\}$ as multisets. There are $m_{r+k}(\nu)$
such $j$, so \eqref{eq:Dmgoal} and Corollary~\ref{cor:Wratio} give
\begin{equation*}
[\hrow\nu]\,\Th_{\mathsf p_k}\hrow\lambda\big|_{t=1}
=(-1)^{2k-1}\,\widehat m_r(\lambda)\,\frac{\qq r}{\qq{r+k}}\,
\qq{k-1}\qbin{r+k}{k}
=\frac{\widehat m_r(\lambda)}{q^k-1}
=[\hrow\nu]\,D_k\hrow\lambda .\qedhere
\end{equation*}
\end{proof}

\begin{rem}\label{rem:genericfail}
Of the two properties of the $D_k$ in Lemma~\ref{lem:Dkder}, only the
derivation law is special to $t=1$. Already $\Th_{\mathsf p_1}$ fails to be a
derivation at generic $t$ since
$\Th_{\mathsf p_1}(\mathsf e_1^2)-2\,\mathsf e_1\Th_{\mathsf p_1}(\mathsf e_1)
={-(1-q)(1-t)\,\mathsf s_{111}}$ and the defect vanishes exactly at $t=1$.
The commutation $[D_j,D_k]=0$ is the specialization of the generic identity
$\Th_f\Th_g=\Th_{fg}=\Th_g\Th_f$.
\end{rem}


\subsection{Specialization}\label{ssec:spec}
Theorem~\ref{thm:Dm} specializes $\Th_{\mathsf p_k}$ one subscript at a
time. Since Lemma~\ref{lem:residue-transfer} gives regularity of the whole
operator, composites specialize as well. This is the form in which
Theorem~\ref{thm:Dm} is used from now on.

\begin{prop}\label{prop:spec}
Let $f\in\Lambda_{\mathbb Q(q)}$. As an operator on $\Lambda^+$, $\Th_f$ is
regular at $t=1$, meaning that its matrix in the row basis is regular there
in each degree, and for every partition $\mu$ we have
\[
\Th_{\mathsf p_\mu}\big|_{t=1}
=D_{\mu_1}\circ\cdots\circ D_{\mu_{\ell(\mu)}}
\qquad\text{on }\Lambda^+_{\mathbb Q(q)} .
\]
Consequently, if $g\in\Lambda^+$ is regular at $t=1$, then so is $\Th_fg$,
with $(\Th_fg)|_{t=1}=\Th_f|_{t=1}\bigl(g|_{t=1}\bigr)$. In particular, for
such $g$, every specialization occurring in Proposition~\ref{prop:master}
exists, and the families
$\{\Th_{\mathsf e_\lambda}g|_{t=1}\}_\lambda$ and
$\{\Th_{\mathsf p_\mu}g|_{t=1}\}_\mu$ determine one another.
\end{prop}
\begin{proof}
In each degree, the matrix of $\Th_{\mathsf p_k}$ in the row basis is regular
at $t=1$ with value $D_k$, by Lemma~\ref{lem:residue-transfer} and
Theorem~\ref{thm:Dm}. By \ref{it:Rring},
$\Th_{\mathsf p_\mu}=\Th_{\mathsf p_{\mu_1}}\cdots\Th_{\mathsf p_{\mu_{\ell(\mu)}}}$
gives the displayed formula, and every $f$ is a finite
$\mathbb Q(q)$-linear combination of the $\mathsf p_\mu$. The consequences
follow.
\end{proof}

\begin{rem}\label{rem:HS}
The derivation law extends to a product rule for every subscript.
Write $\Delta f=f[X+Y]=\sum f_{(1)}\otimes f_{(2)}$ for the coproduct.\footnote{This is not to be confused with the Delta operator from Macdonald theory.}
Recall that the power sums are primitive and that Theorem~\ref{thm:Dm} makes each $\Th_{\mathsf p_k}|_{t=1}$ a derivation.
Since $f\mapsto\Th_f$ and $\Delta$ are algebra maps we obtain
\[
\Th_f(gh)\big|_{t=1}=\sum \Th_{f_{(1)}}g\big|_{t=1}\cdot\Th_{f_{(2)}}h\big|_{t=1}
\qquad(f\in\Lambda_{\mathbb Q(q)},\ g,h\in\Lambda^+_{\mathbb Q(q)}).
\]
Let $u$ be a formal variable and set $\mathsf E[uX]=\sum_{k\ge0}u^k\mathsf e_k$.
The corresponding operator is
\[
\Th_{\mathsf E[uX]}\big|_{t=1}
=\exp\Big(\sum_{r\ge1}(-1)^{r-1}u^rD_r/r\Big).
\]
This exponential of a derivation defines a ring automorphism of
$\Lambda^+_{\mathbb Q(q)}[[u]]$ and is the counterpart of the
multiplicativity of $\nabla$ at $t=1$ discussed in the introduction.
\end{rem}
\section{A derivation-friendly family of symmetric functions}\label{sec:gamma}

This section introduces the composition-indexed family $\Gamma_\alpha$ on
which both LLT expansions rest. We will show that each $D_k\Gamma_\alpha$
is a nonnegative integral combination of members of this family.
Throughout this section and the next,
$\alpha=(\alpha_0,\alpha_1,\dots,\alpha_s)$ denotes a nonempty composition,
indexed from $0$ because the first part has a distinguished role.

\begin{defn}\label{def:Gamma}
Define $\Gamma_\alpha\in\Lambda_{\mathbb Q(q)}$ recursively by
\[
\Gamma_{(n)}=\hrow n,
\qquad
(q^b-1)\,\Gamma_{\alpha\cdot(b)}=\Gamma_{\alpha\odot(b)}-\Gamma_\alpha\,\hrow b
\qquad(b\ge1).
\]
\end{defn}
It is clear that the recursion in Definition~\ref{def:Gamma} determines the family by induction on the number of parts. 
It is also convenient to read the recursion as a rule for \emph{multiplying by a row
function}: for every nonempty composition $\beta$ and every $k\ge1$, we have that
\[
\Gamma_\beta\,\hrow k
=\Gamma_{\beta\odot(k)}-(q^k-1)\,\Gamma_{\beta\cdot(k)}.
\]
Solving
the recursion in general expresses $\Gamma_\alpha$ in the row basis
with coefficients in $\mathbb Q(q)$, but signs and denominators appear at
every step, and positivity is not obvious.



\begin{eg}\label{eg:gammathree}
As an example which is suggestive of what is to come, we have
\[
\Gamma_{(a,b)}=\frac{\hrow{a+b}-\hrow{(a,b)}}{q^b-1}=D_b\hrow a .
\]
For something more involved consider 
$\alpha=(1,2,1)$. Then we have
\[
\Gamma_{(1,2,1)}
=\frac{\hrow4}{(q^{3}-1)(q-1)}
-\frac{(q^{2}+2q+2)\,\hrow{(3,1)}}{(q^{3}-1)(q^{2}-1)}
+\frac{\hrow{(2,1,1)}}{(q^{2}-1)(q-1)}=\mathsf s_{22}+(q+1)\,\mathsf s_{211}+(q^{2}+q)\,\mathsf s_{1111}\,.
\]
Theorem~\ref{thm:bfsdfs} will explain both the cancellation and the
positivity in the Schur basis.
\end{eg}

We now compute $D_k\Gamma_\alpha$ with the aim of showing that the result is a nonnegative integral combination of the $\Gamma_{\beta}$.
To this end, the following lemma is helpful.

\begin{lem}\label{lem:tworec}
    For a nonempty composition $\alpha$ and integers $b,k\geq 1$ the following hold.
    \begin{align*}
    (q^b-1)D_k\Gamma_{\alpha\cdot(b)}
&=D_k\Gamma_{\alpha\odot(b)}-\bigl(D_k\Gamma_\alpha\bigr)\hrow b
-\Gamma_\alpha\Gamma_{(b,k)},\\
    \Gamma_{\alpha\odot(b)\cdot(k)}
&=q^b\Gamma_{\alpha\cdot(b+k)}
+(q^b-1)\Gamma_{\alpha\cdot(b)\cdot(k)}
+\Gamma_\alpha\Gamma_{(b,k)}.
    \end{align*}
\end{lem}
\begin{proof}
    The first relation follows by applying the derivation $D_k$ to the recursion in Definition~\ref{def:Gamma} and using the fact that $D_k\hrow b=\Gamma_{(b,k)}$. The second is slightly more intricate.

    We expand $\Gamma_{\alpha\odot(b)}\hrow k$ in two ways. On the one hand we get
    \begin{align}\label{eq:one_way}
        \Gamma_{\alpha\odot(b)}\hrow k
=\Gamma_{\alpha\odot(b+k)}-(q^k-1)\Gamma_{\alpha\odot(b)\cdot(k)} 
=\Gamma_\alpha\hrow{b+k}+(q^{b+k}-1)\Gamma_{\alpha\cdot(b+k)}
-(q^k-1)\Gamma_{\alpha\odot(b)\cdot(k)}.
    \end{align}
On the other hand we have
\begin{align}
    \Gamma_{\alpha\odot(b)}\hrow k
    &=\Gamma_\alpha\hrow
b\hrow k+(q^b-1)\Gamma_{\alpha\cdot(b)}\hrow k\nonumber\\
&=
\Gamma_\alpha\hrow {b+k}-(q^k-1)\Gamma_{\alpha}\Gamma_{(b,k)}
+(q^b-1)\Gamma_{\alpha\cdot(b+k)}-(q^b-1)(q^k-1)\Gamma_{\alpha\cdot (b)\cdot (k)}.\label{eq:two_way}
\end{align}
Equating the right hand sides in~\eqref{eq:one_way} and~\eqref{eq:two_way} and collecting the target term $\Gamma_{\alpha\odot (b)\cdot (k)}$ on one side we get
\begin{align*}
    (q^{k}-1)\Gamma_{\alpha\odot (b)\cdot (k)}=q^b(q^{k}-1)\Gamma_{\alpha \cdot (b+k)}+(q^{k}-1)\Gamma_{\alpha}\Gamma_{(b,k)}+(q^{b}-1)(q^k-1)\Gamma_{\alpha\cdot (b)\cdot (k)}.
\end{align*}
Cancelling $q^{k}-1\neq 0$ yields the second claim.
\end{proof}

\begin{defn}
    For $\alpha=(\alpha_0,\dots,\alpha_s)$ a nonempty composition and $k\geq 1$ define the multiset
    \[
    Q_k(\alpha)\coloneqq\{(\alpha_0,\dots,\alpha_j+k,\dots,\alpha_s) \mid 1\leq j\leq s\}\sqcup \{(\alpha_0,\dots,\alpha_{j},k,\alpha_{j+1},\dots,\alpha_s) \mid 0\leq j\leq s\}.
    \]
    Now define the symmetric function  
    \[
    \mathsf R_k(\alpha)\;\coloneqq\;\sum_{\gamma\in Q_{k}(\alpha)} \Gamma_{\gamma}.
    \]
\end{defn}
Note that the first multiset consists of the compositions obtained by adding $k$ to one of $\alpha_1,\dots,\alpha_s$
and the second of those obtained by inserting $k$ as a new part immediately after one of
$\alpha_0,\dots,\alpha_s$. Neither operation alters the initial part
$\alpha_0$, and repeated compositions in $Q_k(\alpha)$ contribute with
their multiplicities.
\begin{eg}
    For
$\alpha=(2,1,1)$ and $k=1$, we have that
\[
\mathsf R_1(\alpha)=\Gamma_{(2,1+\underline{1},1)}+\Gamma_{(2,1,1+\underline{1})}+\Gamma_{(2,\underline{1},1,1)}+\Gamma_{(2,1,\underline{1},1)}+\Gamma_{(2,1,1,\underline{1})}.
\]    
\end{eg}

\begin{prop}\label{prop:insertion}
For $\alpha=(\alpha_0,\dots,\alpha_s)$ a nonempty composition and $k\geq 1$ we have
\[
D_k\Gamma_\alpha=\mathsf R_k(\alpha).
\]
\end{prop}
\begin{proof}
We argue by induction on the number of parts of $\alpha$. 
The case where $\alpha$ has a single part is easily checked. 
A composition with at least two parts has the form $\alpha\cdot(b)$ with
$b\ge1$.
By the first
identity in Lemma~\ref{lem:tworec} and the induction hypothesis applied to
$\alpha$ and to $\alpha\odot(b)$, it therefore suffices to prove that 
\[
(q^b-1)\mathsf R_k(\alpha\cdot(b))
=\mathsf R_k(\alpha\odot(b))-\mathsf R_k(\alpha)\hrow b
-\Gamma_\alpha\Gamma_{(b,k)}.
\]
To this end observe that 
\[
\mathsf R_k(\alpha\cdot(b))=
\sum_{\gamma\in Q_{k}(\alpha)}\Gamma_{\gamma\cdot(b)}
+\Gamma_{\alpha\cdot(b+k)}+\Gamma_{\alpha\cdot(b)\cdot(k)}.
\]
Multiplying by $(q^{b}-1)$ and applying Definition~\ref{def:Gamma} to each $\Gamma_{\gamma\cdot (b)}$ one obtains
\begin{align}\label{eq:one_step_closer}
    (q^b-1)\mathsf R_k(\alpha\cdot(b))=
\sum_{\gamma\in Q_{k}(\alpha)}\Gamma_{\gamma\odot(b)}-\mathsf R_k(\alpha)\hrow b
+(q^b-1)\Gamma_{\alpha\cdot(b+k)}+(q^b-1)\Gamma_{\alpha\cdot(b)\cdot(k)}.
\end{align}
Next observe that
\[
\sum_{\gamma\in Q_k(\alpha)}\Gamma_{\gamma\odot(b)}
=\mathsf R_k(\alpha\odot(b))-\Gamma_{\alpha\odot(b)\cdot(k)}
 +\Gamma_{\alpha\cdot(b+k)}.
\]
Substituting in \eqref{eq:one_step_closer} and then applying the second
identity in Lemma~\ref{lem:tworec} yields the claim.
\end{proof}


\begin{rem}\label{rem:Enk}
The family $\Gamma_\alpha$ refines a classical decomposition of $\nabla\mathsf e_n$
at $t=1$. Following \cite[\S5]{HMZ} define $E_{n,k}\in\Lambda_{\mathbb Q(q)}$ by
\[
\mathsf e_n\Big[X\,\frac{1-z}{1-q}\Big]=\sum_{k=1}^{n}\frac{(z;q)_k}{(q;q)_k}\,E_{n,k},
\]
where $(z;q)_k=(1-z)(1-qz)\cdots(1-q^{k-1}z)$. Setting $z=q$ gives
$\mathsf e_n=\sum_k E_{n,k}$. By \cite[Cor.~5.5]{HMZ} and the compositional shuffle
theorem \cite[Thm.~7.5]{CM}, $\nabla E_{n,k}$ is the sum of
$t^{\mathrm{area}(\pi)}F_\pi(X;q)$ over the Dyck paths $\pi$ of size $n$ that return
to the diagonal $k$ times, as conjectured in \cite[Conj.~6.25]{Hag08}. Here $F_\pi$ is
the path symmetric function of \cite[(6.8)]{Hag08}. Under the encoding of
Remark~\ref{rem:normalization} $F_\pi$ becomes $\LLT(T)$ and the returns of $\pi$ to
the diagonal correspond to the children of the root of $T$. At $t=1$
Theorem~\ref{thm:bfsdfs} therefore gives for $1\le k\le n$
\[
\nabla E_{n,k}\big|_{t=1}=\sum_{\substack{\alpha\vDash n\\ \alpha_0=k}}\Gamma_\alpha .
\]
Summing over $k$ recovers the Shuffle theorem at $t=1$ in the form
$\nabla\mathsf e_n|_{t=1}=\sum_{\alpha\vDash n}\Gamma_\alpha$. This refinement differs
from that of \cite{HMZ}, which records where $\pi$ touches the diagonal. For instance
every Dyck path contributing to $\nabla E_{4,1}$ touches the diagonal only at its
endpoints, whereas
$\nabla E_{4,1}|_{t=1}=\Gamma_{(1,3)}+\Gamma_{(1,2,1)}+\Gamma_{(1,1,2)}+\Gamma_{(1,1,1,1)}$.
We are grateful to Mike Zabrocki for discussions that lead to this observation.
\end{rem}

\section{An LLT expansion of \texorpdfstring{$\Gamma_\alpha$}{Gamma} over plane trees}\label{sec:gammatree}
The central theorem of this section is the following LLT expansion, which we shall establish after developing the necessary combinatorics.
\begin{thm}\label{thm:bfsdfs}
For every nonempty composition $\alpha$ we have
\[
\sum_{\bfsdeg(T)=\alpha}\LLT(T)\;=\;\Gamma_\alpha.
\]
\end{thm}
To illustrate we revisit the case $\alpha=(1,2,1)$.
There are exactly two trees with BFS degree word equaling $(1,2,1)$, as shown in Figure~\ref{fig:gammatrees}.
Their LLT polynomials are respectively
\[
q^{2}\,\mathsf s_{1111}+q\,\mathsf s_{211}
\qquad\text{and}\qquad
\mathsf s_{22}+\mathsf s_{211}+q\,\mathsf s_{1111}.
\]
Their sum is the Schur expansion computed in Example~\ref{eg:gammathree}.

\begin{figure}[ht]
\centering
\includegraphics{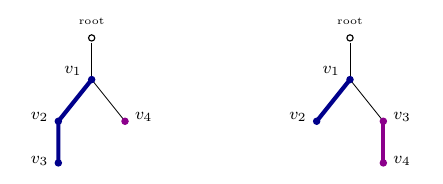}
\caption{The two trees with BFS degree word $(1,2,1)$.}
\label{fig:gammatrees}
\end{figure}

Fix a plane rooted tree $T$ with $\bfsdeg(T)=\alpha$ and let $u_0$ be its
last non-leaf in BFS order, with the root included among the non-leaves.
Write $h$ for the depth of $u_0$ and $\mathrm{Lv}(T)$ for the set of
leaves occurring after $u_0$ in BFS order.

An active vertex of depth greater than $h+1$ would have a non-leaf parent
after $u_0$ in BFS order, so $T$ has depth at most $h+1$.
All vertices at depth $h+1$ are leaves, and at least one exists because
$u_0$ has a child. Thus $\mathrm{Lv}(T)$ consists of the depth-$(h+1)$ leaves
together with the active vertices of depth $h$ strictly to the right of $u_0$.
We list
$\mathrm{Lv}(T)=\{u_1,\dots,u_L\}$ in BFS order.

For $b\ge1$ let $T_0$ be obtained by appending $b$ new children to $u_0$
after all of its existing children. For each $1\le j\le L$ instead give
the leaf $u_j$ exactly $b$ children to obtain $T_j$, as in
Figure~\ref{fig:blockinsert}. In each case the original vertices and child
orders are unchanged, and we call the $b$ new vertices the \emph{block}.
The resulting BFS degree words are
\[
  \bfsdeg(T_0)=\alpha\odot(b) \qquad\text{and}\qquad
  \bfsdeg(T_j)=\alpha\cdot(b) \quad\text{ for }1\le j\le L.  
\]
These BFS degree words give the two composition indices in the recursion
of Definition~\ref{def:Gamma}.

\begin{figure}[ht]
\centering
\includegraphics{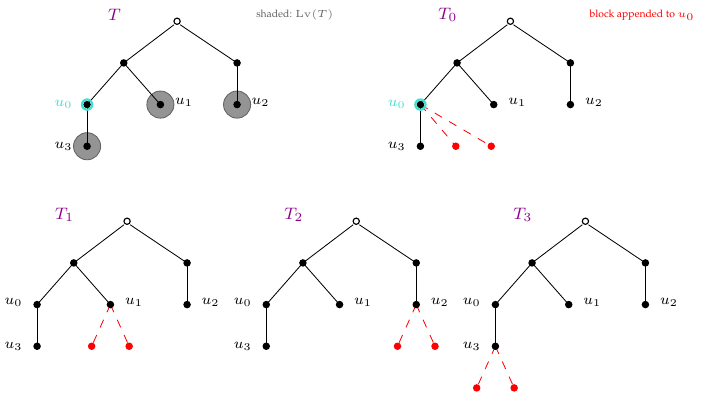}
\caption{The trees $T$ and $T_0$ through $T_3$ contributing to \eqref{eq:blockinsert} for $b=2$.}
\label{fig:blockinsert}
\end{figure}

Now fix a valid coloring $\mathbf c$ of the active vertices of $T$
and a word $\mathbf w=(w_1,\dots,w_b)$ of positive integers.
For an active vertex $v=v_i$ we write $c_v=c_i$.
We use $\mathbf w$ to color the block of each $T_j$ from left to right,
though the resulting colorings need not be valid.
Set
\[
m_j=\#\{i\le b:\ w_i>c_{u_j}\}\quad(1\le j\le L),
\qquad
d_j=\sum_{i>j}m_i\quad(0\le j\le L),
\]
so that $d_L=0$ and $d_{j-1}=m_j+d_j$. 
Both the $m_j$ and the $d_j$ depend only on the content of $\mathbf w$,
while its arrangement enters only through $\mathrm{coinv}(\mathbf w)$
and the first letter $w_1$.

We claim that in $T_j$ a block vertex forms a pair with an old vertex $v$
exactly when $v=u_i$ for some $i>j$ and the block color exceeds $c_v$.
Recall that $T$ has depth at most $h+1$, and note that DFS and BFS order
agree on vertices of equal depth. If $j=0$ or $u_j$ has depth $h$, the
block lies at depth $h+1$. Every old vertex of depth $h+1$ has its parent
no later than $u_0$, so it precedes the block in DFS order and forms primary
pairs with it. The old vertices of depth $h$ that follow the block are those
to the right of $u_j$, and they form secondary pairs with it. If $u_j$ has
depth $h+1$, the block lies at depth $h+2$, and only the old vertices of
depth $h+1$ to the right of $u_j$ form pairs with it, all secondary. In
every case a pair occurs exactly when the block color is the larger one, and
by the BFS listing of $\mathrm{Lv}(T)$ the old vertices involved are
exactly the $u_i$ with $i>j$.

Hence the old pairs, the pairs within the block and the mixed pairs
contribute $\mathrm{dinv}(\mathbf c)$, $\mathrm{coinv}(\mathbf w)$ and
$\sum_{i>j}m_i=d_j$ respectively. The only new rise in depth is from $u_j$
to the first block vertex when $j\ge1$, and we obtain the following.
\begin{enumerate}[label=\textup{(\roman*)},ref=\textup{(\roman*)}]
\item\label{it:blockvalid} The resulting coloring of $T_0$ is always valid,
and for $1\le j\le L$ the resulting coloring of $T_j$ is valid if and only
if $w_1>c_{u_j}$.
\item\label{it:blockdinv} In either case the diagonal inversion number of
the resulting coloring of $T_j$ is
\begin{equation}\label{eq:blockdinv}
\mathrm{dinv}(\mathbf c)+\mathrm{coinv}(\mathbf w)+d_j .
\end{equation}
\end{enumerate}
Figure~\ref{fig:blockdinv} carries out this count on an example.

\begin{figure}[ht]
\centering
\includegraphics{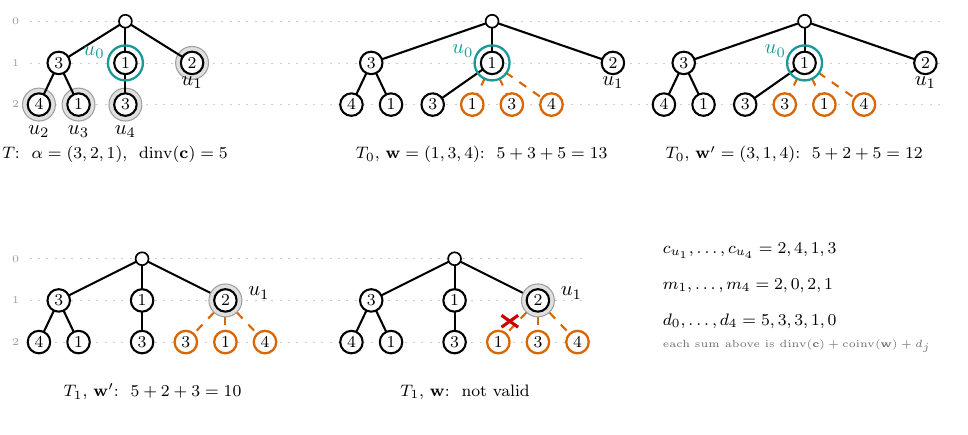}
\caption{The count \eqref{eq:blockdinv} for $b=3$, with one valid
coloring $\mathbf c$ of $T$ and the two block words $\mathbf w=(1,3,4)$ and
$\mathbf w'=(3,1,4)$.}
\label{fig:blockdinv}
\end{figure}

\begin{proof}[Proof of Theorem~\ref{thm:bfsdfs}]
The following identity is the key:
\begin{equation}\label{eq:blockinsert}
\LLT(T_0)=(q^b-1)\sum_{j=1}^{L}\LLT(T_j)+\LLT(T)\,\hrow b.
\end{equation}
 Let us first see how it implies the theorem.
For a one-part composition $(n)$, the unique tree is the $n$-tree whose
components are single leaves and its
LLT polynomial is $\hrow n$ by \eqref{eq:rowcoinv}.
Now fix $b\ge1$ and sum \eqref{eq:blockinsert} over the trees with BFS degree
word $\alpha$, forming $T_0,\dots,T_L$ from each tree. The map
$T\mapsto T_0$ is a bijection onto
the trees with BFS degree word $\alpha\odot(b)$.
Its inverse removes the final $b$ children of the last non-leaf,
all of which are leaves because there is no later non-leaf.
Similarly, $(T,j)\mapsto T_j$ is a bijection
from the pairs with $1\le j\le L$ onto the trees with BFS degree word
$\alpha\cdot(b)$.
Its inverse removes the $b$ children of the last non-leaf $u_j$, leaving
$u_j$ as a leaf after the last non-leaf of the resulting tree in BFS order.
The tree sums therefore satisfy the recursion in Definition~\ref{def:Gamma},
and induction on the number of parts proves the theorem.

It remains to prove \eqref{eq:blockinsert}.
By \ref{it:blockvalid}, \eqref{eq:blockdinv} and \eqref{eq:rowcoinv},
\begin{align*}
\LLT(T_0)
&=\sum_{\mathbf c\in\mathrm{Col}(T_0)}\wt(\mathbf c)
=\sum_{\mathbf c'\in\mathrm{Col}(T)}\wt(\mathbf c')
\sum_{|\beta|=b}q^{d_0}\,\mathbf x^{\beta}\,C(\beta)\\
&=\LLT(T)\,\hrow b
+\sum_{\mathbf c'\in\mathrm{Col}(T)}\wt(\mathbf c')
\sum_{|\beta|=b}\bigl(q^{d_0}-1\bigr)\mathbf x^{\beta}\,C(\beta).
\end{align*}
Since $d_L=0$ and $d_{j-1}-d_j=m_j$, Lemma~\ref{lem:firstletter} with
$r=c_{u_j}$ gives
\[
\bigl(q^{d_0}-1\bigr)C(\beta)
=\sum_{j=1}^{L}q^{d_j}\bigl(q^{m_j}-1\bigr)C(\beta)
=(q^b-1)\sum_{j=1}^{L}q^{d_j}\,C(\beta,c_{u_j}).
\]
Substituting, and applying \ref{it:blockvalid} and \eqref{eq:blockdinv}
once more,
\begin{equation*}
\LLT(T_0)-\LLT(T)\,\hrow b
=(q^b-1)\sum_{j=1}^{L}\sum_{\mathbf c'\in\mathrm{Col}(T)}\wt(\mathbf c')
\sum_{|\beta|=b}q^{d_j}\,\mathbf x^{\beta}\,C(\beta,c_{u_j})
=(q^b-1)\sum_{j=1}^{L}\LLT(T_j).\qedhere
\end{equation*}
\end{proof}


\section{LLT expansions of Theta operators on single rows}\label{sec:expansions}

By Propositions~\ref{prop:spec} and~\ref{prop:insertion},
$\Th_{\mathsf p_\mu}\hrow a|_{t=1}=D_{\mu_1}\cdots D_{\mu_r}\hrow a$
is a sum of $\Gamma$ functions.
Theorem~\ref{thm:bfsdfs} expands each of these in vertical-strip LLT polynomials,
so it remains to organize the resulting terms by ordered set partitions.


\begin{defn}
    An
\emph{ordered set partition} of $[r]$ is a tuple $\mathcal B=(B_1,\dots,B_s)$
of nonempty pairwise disjoint subsets with union $[r]$. Write $\OP_r$ for
the set of ordered set partitions.
\end{defn}
Given a composition $\mu=(\mu_1,\dots,\mu_r)$ and $B\subseteq [r]$ we set $\mu(B)=\sum_{i\in B}\mu_i$. 
For $a\ge1$, define
\begin{equation}\label{eq:setcomp}
\mathsf K^{(a)}_\mu
\;\coloneqq\;
\sum_{\mathcal B=(B_1,\dots,B_s)\in\OP_r}
\Gamma_{(a,\,\mu(B_1),\dots,\mu(B_s))},
\qquad
\mathsf K^{(a)}_\varnothing=\Gamma_{(a)}=\hrow a.
\end{equation}
The function $\mathsf K^{(a)}_\mu$ depends only on the multiset of entries of $\mu$
and hence on the underlying partition.
The ordered set partitions record the successive insertions of
Proposition~\ref{prop:insertion}.
The block containing $i$ records which part receives $\mu_i$,
while the order of the blocks records the positions of the inserted parts.
The next result establishes this interpretation and proves
Theorem~\ref{thm:mainLLT} from the introduction.

\begin{thm}\label{thm:LLTseeded}
For every partition $\mu=(\mu_1,\dots,\mu_r)$ and every $a\ge1$, we have 
\[
\Th_{\mathsf p_\mu}\hrow a\big|_{t=1}
=\mathsf K^{(a)}_\mu
=\sum_{\mathcal B=(B_1,\dots,B_s)\in\OP_r}
  \ \sum_{\bfsdeg(T)=(a,\mu(B_1),\dots,\mu(B_s))}
  \LLT(T).
\]
\end{thm}
\begin{proof}
We first show that for every $k\ge1$, we have
\begin{equation}\label{eq:seeded}
D_k\,\mathsf K^{(a)}_\mu=\mathsf K^{(a)}_{(\mu_1,\dots,\mu_r,k)} ,
\qquad\text{hence}\qquad
\mathsf K^{(a)}_\mu=D_{\mu_1}\!\cdots D_{\mu_r}\hrow a .
\end{equation}
Given $(B_1,\dots,B_s)\in \OP_r$, we obtain $2s+1$ elements of $\OP_{r+1}$
by adjoining $r+1$ to one of the $s$ blocks or inserting it as a singleton
in one of the $s+1$ positions.
Every element of $\OP_{r+1}$ has a unique predecessor obtained by deleting
$r+1$ and discarding its block if it becomes empty.
Write $\nu=(\mu_1,\dots,\mu_r,k)$, so that
$\nu(B\cup\{r+1\})=\mu(B)+k$ and $\nu(\{r+1\})=k$.
Applying Proposition~\ref{prop:insertion} to each term of \eqref{eq:setcomp} then gives
\[
D_k\,\mathsf K^{(a)}_\mu
=\sum_{(B_1,\dots,B_s)\in\OP_r}\mathsf R_k\big(a,\mu(B_1),\dots,\mu(B_s)\big).
\]
The terms of $\mathsf R_k(a,\mu(B_1),\dots,\mu(B_s))$ either add $k$ to some
$\mu(B_i)$ or insert $k$ as a new part in one of the $s+1$ positions after $a$.
These correspond to adjoining $r+1$ to $B_i$ and to inserting $\{r+1\}$ as a
singleton block in that position, so the right-hand side equals $\mathsf{K}^{(a)}_{\nu}$.
Iterating from
$\mathsf K^{(a)}_\varnothing=\hrow a$ gives
$\mathsf K^{(a)}_\mu=D_{\mu_r}\cdots D_{\mu_1}\hrow a$.
We may reverse the order of the $D_k$ because they commute by Lemma~\ref{lem:Dkder}.

Applying Theorem~\ref{thm:bfsdfs} to each term of \eqref{eq:setcomp} shows that the double sum in the statement is $\mathsf K^{(a)}_\mu$.
The first equality follows from \eqref{eq:seeded} and Proposition~\ref{prop:spec}.
\end{proof}


\begin{eg}\label{eg:Ksmall}
Take $a=1$ and $\mu=(2,1)$. The ordered set partitions $(\{1\},\{2\})$,
$(\{2\},\{1\})$ and $(\{1,2\})$ of $[2]$ give the compositions $(1,2,1)$,
$(1,1,2)$ and $(1,3)$. The two trees with BFS degree word $(1,2,1)$ are those
of Figure~\ref{fig:gammatrees}, and the remaining two compositions each
index a single tree, shown in Figure~\ref{fig:Ksmall}. Summing LLT
polynomials by BFS degree word as in Theorem~\ref{thm:bfsdfs} gives
$\Gamma_{(1,2,1)}$ as in Example~\ref{eg:gammathree}, together with
\[
\Gamma_{(1,3)}=q^{3}\,\mathsf s_{1111}+(q^{2}+q)\,\mathsf s_{211}+q\,\mathsf s_{22}+\mathsf s_{31},
\qquad
\Gamma_{(1,1,2)}=q\,\mathsf s_{1111}+\mathsf s_{211}.
\]
Hence
\[
\Th_{\mathsf p_{(2,1)}}\hrow 1\big|_{t=1}
=\Gamma_{(1,2,1)}+\Gamma_{(1,1,2)}+\Gamma_{(1,3)}
=(q^{3}+q^{2}+2q)\,\mathsf s_{1111}+(q^{2}+2q+2)\,\mathsf s_{211}+(q+1)\,\mathsf s_{22}+\mathsf s_{31}.
\]
\end{eg}

\begin{figure}[ht]
\centering
\includegraphics{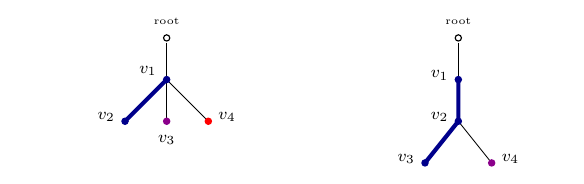}
\caption{The trees with BFS degree words $(1,3)$ and $(1,1,2)$.}
\label{fig:Ksmall}
\end{figure}



We now turn to elementary subscripts, which we will record in the auxiliary
alphabet $Y$ by sibling-increasing labels.

\begin{defn}\label{def:labeltree}
Let $T$ be an $a$-tree with $n+a$ active vertices, and call the $a$
children of the root its \emph{component roots}. A \emph{sibling-increasing
labeling} of $T$ assigns a positive integer to each of the $n$ active
vertices that are not component roots, in such a way that the labels of the
children of a common parent increase strictly from left to right. The
\emph{content} of a labeling is the weak composition
$\beta=(\beta_1,\beta_2,\dots)$ of $n$ in which $\beta_i$ counts the vertices
labeled $i$, and $\LT^{(a)}_\beta$ is the set of sibling-increasing labeled
$a$-trees of content $\beta$.
\end{defn}

Figure~\ref{fig:labeltrees} shows an element of $\LT^{(3)}_{(2,2,2)}$.

\begin{figure}[ht]
\centering
\includegraphics{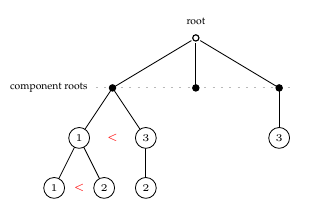}
\caption{A sibling-increasing labeling of a $3$-tree with nine active
vertices.}
\label{fig:labeltrees}
\end{figure}

Definitions~\ref{def:coloring} and~\ref{def:llt} use only depths and DFS order.
For $S\in\LT^{(a)}_\beta$ we therefore write $\LLT(S)$ for the LLT polynomial
of the underlying $a$-tree.
Colors and labels are independent data on the same tree.

\begin{prop}\label{prop:Yexp}
Fix $a\ge1$ and $n\ge1$. In $\widehat{\Lambda(X)\otimes\Lambda(Y)}$ we have
\[
\sum_{\beta}\mathbf y^{\beta}\!\!\sum_{S\in\LT^{(a)}_\beta}\!\!\LLT(S;X)
\;=\;\sum_{T}\mathsf e_{\bfsdeg(T)^\flat}(Y)\,\LLT(T;X)
\;=\;\sum_{\lambda\vdash n}\mathsf m_\lambda(Y)\,\Upsilon^{(a)}_\lambda(X),
\]
where the index $\beta$ runs over the weak compositions of $n$, the tree
$T$ runs over
the $a$-trees with $n+a$ active vertices, and $\bfsdeg(T)^\flat$ is
the BFS degree word with its initial entry $a$ deleted.
\end{prop}
\begin{proof}
Let $v_1,\dots,v_s$ be the non-leaf active vertices of $T$ in BFS
order and $d_j$ the out-degree of $v_j$, so that
$\bfsdeg(T)^\flat=(d_1,\dots,d_s)$. The labelings of $T$ have generating
function $\prod_j\mathsf e_{d_j}(Y)=\mathsf e_{\bfsdeg(T)^\flat}(Y)$.
Multiplying by $\LLT(T;X)$ and summing over $T$ gives the first equality.

For the second, index the parts of $\mu$ by $[r]$ as in Theorem~\ref{thm:LLTseeded}.
The expansion of power sums in the monomial basis is classically counted by
ordered set partitions \cite[Prop.~7.7.1]{StanEC2}. Dualizing and applying
$\omega$, we obtain
\[
[\mathsf p_\mu]\,\mathsf e_{\bfsdeg(T)^\flat}
=\varepsilon_\mu z_\mu^{-1}
\#\bigl\{(B_1,\dots,B_s)\in\OP_r:\mu(B_j)=d_j
\ \text{ for }1\le j\le s\bigr\}.
\]
Summing over $T$, Theorem~\ref{thm:LLTseeded} identifies the
coefficient of $\mathsf p_\mu(Y)$ in the middle sum with
$\varepsilon_\mu z_\mu^{-1}\mathsf K^{(a)}_\mu
=\varepsilon_\mu z_\mu^{-1}\Th_{\mathsf p_\mu}\hrow a|_{t=1}$. The middle
sum is therefore the power-sum side of Proposition~\ref{prop:master} at $g=\hrow a$,
specialized at $t=1$ as Proposition~\ref{prop:spec} permits and read in
$Y$-degree $n$. By Proposition~\ref{prop:master} it equals the monomial side,
which after the same specialization is
$\sum_{\lambda\vdash n}\mathsf m_\lambda(Y)\,\Upsilon^{(a)}_\lambda(X)$.
\end{proof}

Deleting the zero parts of $\beta$ changes the set $\LT^{(a)}_\beta$ but not
the sum $\sum_{S\in\LT^{(a)}_\beta}\LLT(S)$. The following may therefore be
stated for compositions.

\begin{thm}\label{thm:lltrow}
For every $a\ge1$ and every composition $\beta\vDash n$, we have
\[
\Upsilon^{(a)}_\beta=\Th_{\mathsf e_\beta}\hrow a\big|_{t=1}
\;=\;\sum_{S\in\LT^{(a)}_\beta}\LLT(S).
\]
\end{thm}
\begin{proof}
Take the coefficient of $\mathbf y^{\beta}$ in Proposition~\ref{prop:Yexp}. On the left it is
$\sum_{S\in\LT^{(a)}_\beta}\LLT(S)$. On the right, $\mathbf y^{\beta}$
occurs in $\mathsf m_\lambda(Y)$ only for $\lambda$ the decreasing
rearrangement of $\beta$, and there with coefficient $1$, so it is
$\Upsilon^{(a)}_\lambda$. Finally, $\Upsilon^{(a)}_\lambda=\Upsilon^{(a)}_\beta$
because $\mathsf e_\beta=\mathsf e_\lambda$.
\end{proof}

\begin{rem}\label{rem:cauchy}
Applying $\omega_Y$ to Proposition~\ref{prop:Yexp} turns the dual Cauchy
kernel into the ordinary Cauchy kernel
\[
\omega_Y\mathsf E[XY]=\sum_{k\ge0}\mathsf h_k[XY].
\]
Reading the result in different bases of $\Lambda(Y)$ gives expansions for complete,
monomial and Schur subscripts at $t=1$.

The expansion of $\Th_{\mathsf h_\beta}\hrow a|_{t=1}$ is a sum of
$\LLT(S)$ over the labeled $a$-trees of content $\beta$ whose sibling labels
increase \emph{weakly}.
The expansion of $\Th_{\mathsf m_\lambda}\hrow a|_{t=1}$ is the sum of
$\Gamma_{(a)\cdot\alpha}$ over the distinct rearrangements $\alpha$ of $\lambda$.
For Schur subscripts we obtain
\[
\Th_{\mathsf s_\lambda}\hrow a\big|_{t=1}
=\sum_T K_{\lambda,\bfsdeg(T)^\flat}\,\LLT(T),
\]
where $T$ ranges over the $a$-trees with $|\lambda|+a$ active vertices and
$K_{\lambda,\bfsdeg(T)^\flat}$ is the Kostka number.
In particular $\Th_f\hrow a|_{t=1}$ is Schur positive whenever $f$ is.
\end{rem}

\begin{rem}\label{rem:rowproducts}
The product rule in Remark~\ref{rem:HS} extends the single-row expansions to
products of row functions. For every nonempty partition $\mu$ we have
\[
\Th_{\mathsf E[XY]}\hrow\mu\big|_{t=1}
=\prod_i\Th_{\mathsf E[XY]}\hrow{\mu_i}\big|_{t=1}.
\]
This gives expansions of $\Th_{\mathsf e_\lambda}\hrow\mu|_{t=1}$ and
$\Th_{\mathsf p_\nu}\hrow\mu|_{t=1}$ for every such $\mu$.
\end{rem}

\section{Single-row Macdonald cumulants}\label{sec:rowcum}

We now relate the function in Theorem~\ref{thm:LLTseeded} to a single-row
Macdonald cumulant and thereby prove Corollary~\ref{cor:introdolega}.
For $r\ge1$ write $\mathcal P_r$ for the lattice of set partitions of $[r]$
with top element $\hat1$ and M\"obius function $\mu$.
For positive integers $a_1,\dots,a_r$ and $B\subseteq[r]$ set $a(B)=\sum_{i\in B}a_i$.

\begin{defn}[{\cite[Def.~3, (7)]{DKLLT}, \cite[(2), (4)]{DolG}}]\label{def:cumulant}
The \emph{conditional cumulant} and the normalized \emph{$q$-partial
cumulant} associated with $\mathbf a=(a_1,\dots,a_r)$ are
\[
\varkappa(\mathbf a)\coloneqq
\sum_{\sigma\in\mathcal P_r}\mu(\sigma,\hat1)
\prod_{B\in\sigma}\hrow{a(B)},
\qquad
\kappa(\mathbf a)\coloneqq(q-1)^{1-r}\varkappa(\mathbf a).
\]
\end{defn}

The M\"obius coefficient is
$\mu(\sigma,\hat1)=(-1)^{|\sigma|-1}(|\sigma|-1)!$
\cite[Ex.~3.10.4]{StanEC1}. Here $|\sigma|$ denotes the number of blocks of
$\sigma$. Both functions are symmetric in $a_1,\dots,a_r$.
On single rows Do\l{}\k{e}ga's second product $\oplus$ adds row lengths
\cite[\S1]{DolG}. Thus a block $B$ contributes the row function $\hrow{a(B)}$.
The normalized function $\kappa(a_1,\dots,a_r)$ is the Macdonald cumulant
$\kappa((a_1),\dots,(a_r))$ of \cite{DolG}.

\begin{eg}\label{eg:cumsmall}
For $r=2$ the two set partitions of $[2]$ give
$\kappa(a,b)=(\hrow{a+b}-\hrow a\hrow b)/(q-1)$. In particular
$\kappa(1,1)=\mathsf e_2$. For $r=3$ and $a_1=a_2=a_3=1$ we have
\[
\kappa(1,1,1)=\frac{\hrow3-3\hrow2\hrow1+2\hrow1^3}{(q-1)^2}
=\mathsf s_{21}+(q+2)\mathsf s_{111}.
\]
\end{eg}

\begin{thm}\label{thm:Karow}
For all $r\ge1$ and positive integers $a_1,\dots,a_r$ and $k$ we have
\[
(q^k-1)\,D_k\,\varkappa(a_1,\dots,a_r)=\varkappa(a_1,\dots,a_r,k).
\]
Consequently, for every partition $\mu=(\mu_1,\dots,\mu_r)$ with $r\ge0$ and
every $a\ge1$ we have
\[
\frac{\kappa(\mu_1,\dots,\mu_r,a)}{\prod_i[\mu_i]_q}
=D_{\mu_1}\!\cdots D_{\mu_r}\,\hrow a ,
\]
the case $r=0$ reading $\kappa(a)=\hrow a$.
\end{thm}
\begin{proof}
Set $a_{r+1}=k$ and write $\pi_\sigma=\prod_{B\in\sigma}\hrow{a(B)}$.
By Definition~\ref{def:Dk} we have
\[
(q^k-1)D_k\pi_\sigma
=\sum_{B\in\sigma}\pi_{\sigma_B}\;-\;|\sigma|\,\pi_{\sigma_{r+1}},
\]
where $\sigma_B$ is obtained by adding $r+1$ to the block $B$
and $\sigma_{r+1}$ by adding $\{r+1\}$ as a new block.
Every set partition of $[r+1]$ has exactly one of these forms,
and deleting the new element recovers a unique partition $\sigma$.
It remains to compare the M\"obius coefficients in these two cases.

Adding $r+1$ to an existing block preserves the number of blocks,
so $\pi_{\sigma_B}$ already carries the required coefficient.
Adding $r+1$ as a singleton increases the number of blocks by one
and changes the coefficient to $(-1)^{|\sigma|}|\sigma|!$.
This is precisely $-|\sigma|$ times the coefficient of $\sigma$.
Summing over $\sigma\in\mathcal P_r$ therefore gives the M\"obius expansion
of $\varkappa(a_1,\dots,a_r,k)$ and proves the first identity.

For the second identity, start from $\varkappa(a)=\hrow a$ and apply
the first with $k=\mu_r,\dots,\mu_1$ in turn.
Symmetry gives the last equality in
\[
\prod_i(q^{\mu_i}-1)\cdot D_{\mu_1}\cdots D_{\mu_r}\hrow a
=\varkappa(a,\mu_r,\dots,\mu_1)=\varkappa(\mu_1,\dots,\mu_r,a).
\]
Each factor satisfies $q^{\mu_i}-1=(q-1)[\mu_i]_q$,
while $\varkappa(\mu_1,\dots,\mu_r,a)=(q-1)^r\kappa(\mu_1,\dots,\mu_r,a)$ by Definition~\ref{def:cumulant}.
Dividing by $(q-1)^r$ therefore gives the second identity.
\end{proof}

Theorem~\ref{thm:LLTseeded} identifies the right-hand side of the
second identity in Theorem~\ref{thm:Karow} with $\Th_{\mathsf p_\mu}\hrow a|_{t=1}=\mathsf K^{(a)}_\mu$
and hence gives \eqref{eq:introKarow}.

\begin{cor}\label{cor:dolega}
For all row lengths $a_1,\dots,a_r\ge1$, the single-row Macdonald cumulant
$\kappa(a_1,\dots,a_r)$ is a $\mathbb Z_{\ge0}[q]$-linear combination of
vertical-strip LLT polynomials and hence Schur positive.
\end{cor}
\begin{proof}
Sort $a_1,\dots,a_{r-1}$ into a partition $\mu$ and take $a=a_r$ in the
second identity in Theorem~\ref{thm:Karow}. By Theorem~\ref{thm:LLTseeded} the cumulant is then
$\prod_{i<r}[a_i]_q\in\mathbb Z_{\ge0}[q]$
times a nonnegative integral combination of vertical-strip LLT polynomials,
which are Schur positive.
Since the cumulant is independent of $t$, its Schur coefficients also lie
in $\mathbb Z_{\ge0}[q,t]$
as required by \cite[Conj.~1.2]{DolG}.
\end{proof}

In the terminology of \cite{DKLLT} the function $\kappa(a_1,\dots,a_r)$ is
the LLT cumulant of $a_1+\cdots+a_r$ single boxes of equal content with color
multiplicities $a_1,\dots,a_r$.
The cospin and shift corrections vanish because all cells have equal content,
so the three normalizations in \cite{DKLLT} agree.
Theorems~\ref{thm:LLTseeded} and~\ref{thm:Karow} give the expansion in LLT
polynomials asked for at the end of \cite{DKLLT} for arbitrary
color multiplicities.

\begin{rem}\label{rem:DK38}
The LLT graph of $n$ single boxes of equal content is the complete graph
$K_n$. The Schur positivity in Corollary~\ref{cor:dolega} would therefore
also follow from \cite[Thm.~38]{DKLLT} because $K_n$ is a melting lollipop.
That theorem gives a formula for the Schur coefficients of the LLT cumulant
of a melting lollipop with an arbitrary coloring. This formula does not hold
in general: the coefficient of $\mathsf s_{22}$ in $\kappa(1,1,1,1)$ is $q+3$ whereas the formula gives $q+2$.
The proof in \cite{DKLLT} applies the Schur expansion of Huh, Nam and Yoo
\cite{HNY20} to the disjoint union of the restrictions of $K_n$ to the
blocks of a set partition with the vertex labels inherited from $K_n$.
If some block is not an interval then there are $i<j<k$ with $i\sim k$ and
$i\not\sim j$. The labeled disjoint union is then not a unit interval graph and the
expansion of products stated in \cite[Thm.~38]{DKLLT} can fail. For the set
partition $\{\{1,4\},\{2,3\}\}$ of the vertices of $K_4$ the coefficient of
$\mathsf s_{22}$ in $\hrow2\hrow2$ is $q^2+1$ whereas the expansion gives $2q$.
This set partition contributes to $\kappa(1,1,1,1)$. When all blocks are
intervals the argument of \cite{DKLLT} applies. This covers the cumulants
$\kappa(a_1,a_2)$ since their two color classes can be taken to be intervals.
For distinct colors Kowalski gives an LLT-positive expansion and hence Schur
positivity \cite{Kow20}, \cite[Thm.~40]{DKLLT}, \cite[Thm.~1.2]{Kow23}.
\end{rem}


\bibliographystyle{hplain}
\bibliography{main}

@article{CM,
    AUTHOR = {Carlsson, Erik and Mellit, Anton},
     TITLE = {A proof of the shuffle conjecture},
   JOURNAL = {J. Amer. Math. Soc.},
  FJOURNAL = {Journal of the American Mathematical Society},
    VOLUME = {31},
      YEAR = {2018},
    NUMBER = {3},
     PAGES = {661--697},
      ISSN = {0894-0347,1088-6834},
   MRCLASS = {05E10 (05E05 33D52)},
  MRNUMBER = {3787405},
MRREVIEWER = {Tanja\ Stojadinovi\'{c}},
       DOI = {10.1090/jams/893},
       URL = {https://doi.org/10.1090/jams/893},
}

@article{DIV,
    AUTHOR = {D'Adderio, Michele and Iraci, Alessandro and Vanden Wyngaerd,
              Anna},
     TITLE = {Theta operators, refined delta conjectures, and coinvariants},
   JOURNAL = {Adv. Math.},
  FJOURNAL = {Advances in Mathematics},
    VOLUME = {376},
      YEAR = {2021},
     PAGES = {Paper No. 107447, 59},
      ISSN = {0001-8708,1090-2082},
   MRCLASS = {05E05},
  MRNUMBER = {4178919},
MRREVIEWER = {Elizabeth\ M.\ Niese},
       DOI = {10.1016/j.aim.2020.107447},
       URL = {https://doi.org/10.1016/j.aim.2020.107447},
}

@article{DR,
    AUTHOR = {D'Adderio, Michele and Romero, Marino},
     TITLE = {New identities for theta operators},
   JOURNAL = {Trans. Amer. Math. Soc.},
  FJOURNAL = {Transactions of the American Mathematical Society},
    VOLUME = {376},
      YEAR = {2023},
    NUMBER = {8},
     PAGES = {5775--5807},
      ISSN = {0002-9947,1088-6850},
   MRCLASS = {05E05 (20C30)},
  MRNUMBER = {4630759},
MRREVIEWER = {Tanja\ Stojadinovi\'{c}},
       DOI = {10.1090/tran/8911},
       URL = {https://doi.org/10.1090/tran/8911},
}

@article{IR,
    AUTHOR = {Iraci, Alessandro and Romero, Marino},
     TITLE = {Delta and theta operator expansions},
   JOURNAL = {Forum Math. Sigma},
  FJOURNAL = {Forum of Mathematics. Sigma},
    VOLUME = {12},
      YEAR = {2024},
     PAGES = {Paper No. e30, 37},
      ISSN = {2050-5094},
   MRCLASS = {05E05 (05E10)},
  MRNUMBER = {4713626},
MRREVIEWER = {Eric\ S.\ Egge},
       DOI = {10.1017/fms.2024.14},
       URL = {https://doi.org/10.1017/fms.2024.14},
}

@article{DILRVW,
    AUTHOR = {D'Adderio, Michele and Iraci, Alessandro and Le Borgne, Yvan
              and Romero, Marino and Vanden Wyngaerd, Anna},
     TITLE = {Tiered trees and {T}heta operators},
   JOURNAL = {Int. Math. Res. Not. IMRN},
  FJOURNAL = {International Mathematics Research Notices. IMRN},
      YEAR = {2023},
    NUMBER = {24},
     PAGES = {20748--20783},
      ISSN = {1073-7928,1687-0247},
   MRCLASS = {05E05 (05C31)},
  MRNUMBER = {4681271},
MRREVIEWER = {Timothy\ Y.\ Chow},
       DOI = {10.1093/imrn/rnac258},
       URL = {https://doi.org/10.1093/imrn/rnac258},
}

@article{BGHT,
    AUTHOR = {Bergeron, F. and Garsia, A. M. and Haiman, M. and Tesler, G.},
     TITLE = {Identities and positivity conjectures for some remarkable
              operators in the theory of symmetric functions},
      NOTE = {Dedicated to Richard A. Askey on the occasion of his 65th
              birthday, Part III},
   JOURNAL = {Methods Appl. Anal.},
  FJOURNAL = {Methods and Applications of Analysis},
    VOLUME = {6},
      YEAR = {1999},
    NUMBER = {3},
     PAGES = {363--420},
      ISSN = {1073-2772,1945-0001},
   MRCLASS = {05E05 (33D52)},
  MRNUMBER = {1803316},
MRREVIEWER = {Ang\`ele\ M.\ Hamel},
       DOI = {10.4310/MAA.1999.v6.n3.a7},
       URL = {https://doi.org/10.4310/MAA.1999.v6.n3.a7},
}

@article{HHLRU,
    AUTHOR = {Haglund, J. and Haiman, M. and Loehr, N. and Remmel, J. B. and
              Ulyanov, A.},
     TITLE = {A combinatorial formula for the character of the diagonal
              coinvariants},
   JOURNAL = {Duke Math. J.},
  FJOURNAL = {Duke Mathematical Journal},
    VOLUME = {126},
      YEAR = {2005},
    NUMBER = {2},
     PAGES = {195--232},
      ISSN = {0012-7094,1547-7398},
   MRCLASS = {05E10 (05A30 20C30)},
  MRNUMBER = {2115257},
MRREVIEWER = {Edward\ E.\ Allen},
       DOI = {10.1215/S0012-7094-04-12621-1},
       URL = {https://doi.org/10.1215/S0012-7094-04-12621-1},
}

@book{Hag08,
    AUTHOR = {Haglund, James},
     TITLE = {The {$q$},{$t$}-{C}atalan numbers and the space of diagonal
              harmonics},
    SERIES = {University Lecture Series},
    VOLUME = {41},
      NOTE = {With an appendix on the combinatorics of Macdonald
              polynomials},
 PUBLISHER = {American Mathematical Society, Providence, RI},
      YEAR = {2008},
     PAGES = {viii+167},
      ISBN = {978-0-8218-4411-3; 0-8218-4411-3},
   MRCLASS = {05E05 (05A05 05A30 33D52)},
  MRNUMBER = {2371044},
MRREVIEWER = {Michael\ A.\ Zabrocki},
}

@book{StanEC1,
    AUTHOR = {Stanley, Richard P.},
     TITLE = {Enumerative combinatorics. {V}olume 1},
    SERIES = {Cambridge Studies in Advanced Mathematics},
    VOLUME = {49},
   EDITION = {Second},
 PUBLISHER = {Cambridge University Press, Cambridge},
      YEAR = {2012},
     PAGES = {xiv+626},
      ISBN = {978-1-107-60262-5},
   MRCLASS = {05-02 (05A15 06-02)},
  MRNUMBER = {2868112},
}

@book{StanEC2,
    AUTHOR = {Stanley, Richard P.},
     TITLE = {Enumerative combinatorics. {V}olume 2},
    SERIES = {Cambridge Studies in Advanced Mathematics},
    VOLUME = {62},
      NOTE = {With a foreword by Gian-Carlo Rota and appendix 1 by Sergey
              Fomin},
 PUBLISHER = {Cambridge University Press, Cambridge},
      YEAR = {1999},
     PAGES = {xii+581},
      ISBN = {0-521-56069-1; 0-521-78987-7},
   MRCLASS = {05-02 (05A15 05E05 06-02)},
  MRNUMBER = {1676282},
}

@article{GT,
    AUTHOR = {Garsia, A. M. and Tesler, G.},
     TITLE = {Plethystic formulas for {M}acdonald {$q,t$}-{K}ostka
              coefficients},
   JOURNAL = {Adv. Math.},
  FJOURNAL = {Advances in Mathematics},
    VOLUME = {123},
      YEAR = {1996},
    NUMBER = {2},
     PAGES = {144--222},
      ISSN = {0001-8708,1090-2082},
   MRCLASS = {05E05 (33C80 33D80)},
  MRNUMBER = {1420484},
MRREVIEWER = {Mark\ D.\ Haiman},
       DOI = {10.1006/aima.1996.0071},
       URL = {https://doi.org/10.1006/aima.1996.0071},
}

@article{Knop,
    AUTHOR = {Knop, Friedrich},
     TITLE = {Integrality of two variable {K}ostka functions},
   JOURNAL = {J. Reine Angew. Math.},
  FJOURNAL = {Journal f\"{u}r die Reine und Angewandte Mathematik. [Crelle's
              Journal]},
    VOLUME = {482},
      YEAR = {1997},
     PAGES = {177--189},
      ISSN = {0075-4102,1435-5345},
   MRCLASS = {05E05 (33C80 33D80)},
  MRNUMBER = {1427661},
MRREVIEWER = {Mark\ D.\ Haiman},
       DOI = {10.1515/crll.1997.482.177},
       URL = {https://doi.org/10.1515/crll.1997.482.177},
}

@article{Sahi,
    AUTHOR = {Sahi, Siddhartha},
     TITLE = {Interpolation, integrality, and a generalization of
              {M}acdonald's polynomials},
   JOURNAL = {Internat. Math. Res. Notices},
  FJOURNAL = {International Mathematics Research Notices},
      YEAR = {1996},
    NUMBER = {10},
     PAGES = {457--471},
      ISSN = {1073-7928,1687-0247},
   MRCLASS = {05E05 (33C80 33D80)},
  MRNUMBER = {1399411},
MRREVIEWER = {Mark\ D.\ Haiman},
       DOI = {10.1155/S107379289600030X},
       URL = {https://doi.org/10.1155/S107379289600030X},
}

@incollection{KN,
    AUTHOR = {Kirillov, A. N. and Noumi, M.},
     TITLE = {{$q$}-difference raising operators for {M}acdonald polynomials
              and the integrality of transition coefficients},
 BOOKTITLE = {Algebraic methods and {$q$}-special functions ({M}ontr\'{e}al,
              {QC}, 1996)},
    SERIES = {CRM Proc. Lecture Notes},
    VOLUME = {22},
     PAGES = {227--243},
 PUBLISHER = {Amer. Math. Soc., Providence, RI},
      YEAR = {1999},
      ISBN = {0-8218-2026-5},
   MRCLASS = {39A13 (05E05 33D52)},
  MRNUMBER = {1726838},
MRREVIEWER = {Laurent\ Habsieger},
       DOI = {10.1090/crmp/022/13},
       URL = {https://doi.org/10.1090/crmp/022/13},
}

@misc{Kow20,
      AUTHOR = {Kowalski, Maciej},
       TITLE = {{LLT} cumulants of unicellular {Y}oung diagrams, parking
                functions and {S}chur positivity},
        YEAR = {2020},
      EPRINT = {arXiv:2011.15080},
}

@article{GHT,
    AUTHOR = {Garsia, A. M. and Haiman, M. and Tesler, G.},
     TITLE = {Explicit plethystic formulas for {M}acdonald {$q,t$}-{K}ostka
              coefficients},
      NOTE = {The Andrews Festschrift (Maratea, 1998)},
   JOURNAL = {S\'{e}m. Lothar. Combin.},
  FJOURNAL = {S\'{e}minaire Lotharingien de Combinatoire},
    VOLUME = {42},
      YEAR = {1999},
     PAGES = {Art. B42m, 45},
      ISSN = {1286-4889},
   MRCLASS = {05E05 (33D52)},
  MRNUMBER = {1701592},
MRREVIEWER = {Yasmine\ B.\ Sanderson},
}

@article{Dol17,
    AUTHOR = {Do{\l}{\polhk e}ga, Maciej},
     TITLE = {Strong factorization property of {M}acdonald polynomials and
              higher-order {M}acdonald's positivity conjecture},
   JOURNAL = {J. Algebraic Combin.},
  FJOURNAL = {Journal of Algebraic Combinatorics. An International Journal},
    VOLUME = {46},
      YEAR = {2017},
    NUMBER = {1},
     PAGES = {135--163},
      ISSN = {0925-9899,1572-9192},
   MRCLASS = {05E05},
  MRNUMBER = {3666415},
MRREVIEWER = {Frank\ Sottile},
       DOI = {10.1007/s10801-017-0750-x},
       URL = {https://doi.org/10.1007/s10801-017-0750-x},
}

@article{DKLLT,
    AUTHOR = {Do{\l}{\polhk e}ga, Maciej and Kowalski, Maciej},
     TITLE = {L{LT} cumulants and graph coloring},
   JOURNAL = {Electron. J. Combin.},
  FJOURNAL = {Electronic Journal of Combinatorics},
    VOLUME = {29},
      YEAR = {2022},
    NUMBER = {4},
     PAGES = {Paper No. 4.5, 35},
      ISSN = {1077-8926},
   MRCLASS = {05E05 (05C15)},
  MRNUMBER = {4497217},
MRREVIEWER = {Domenico\ Senato},
       DOI = {10.37236/10977},
       URL = {https://doi.org/10.37236/10977},
}

@article{DolG,
    AUTHOR = {Do{\l}{\polhk e}ga, Maciej},
     TITLE = {Macdonald cumulants, {$G$}-inversion polynomials and
              {$G$}-parking functions},
   JOURNAL = {European J. Combin.},
  FJOURNAL = {European Journal of Combinatorics},
    VOLUME = {75},
      YEAR = {2019},
     PAGES = {172--194},
      ISSN = {0195-6698,1095-9971},
   MRCLASS = {05E05},
  MRNUMBER = {3862962},
MRREVIEWER = {Emmanuel\ Jean\ Briand},
       DOI = {10.1016/j.ejc.2018.08.011},
       URL = {https://doi.org/10.1016/j.ejc.2018.08.011},
}

@misc{Kow23,
      AUTHOR = {Kowalski, Maciej},
       TITLE = {A combinatorial formula for {LLT} cumulants of melting
                lollipops in terms of spanning trees},
        YEAR = {2023},
      EPRINT = {arXiv:2301.08933},
}

@misc{DIIP,
      AUTHOR = {D'Adderio, Michele and Interdonato, Giovanni and Iraci,
                Alessandro and Pagaria, Roberto},
       TITLE = {Leaving the {H}all: explicit formulas for {N}egu\c{t}
                operators},
        NOTE = {Preprint},
        YEAR = {2026},
      EPRINT = {arXiv:2608.14836v2},
}

@article{LLT,
    AUTHOR = {Lascoux, Alain and Leclerc, Bernard and Thibon, Jean-Yves},
     TITLE = {Ribbon tableaux, {H}all-{L}ittlewood functions, quantum affine
              algebras, and unipotent varieties},
   JOURNAL = {J. Math. Phys.},
  FJOURNAL = {Journal of Mathematical Physics},
    VOLUME = {38},
      YEAR = {1997},
    NUMBER = {2},
     PAGES = {1041--1068},
      ISSN = {0022-2488,1089-7658},
   MRCLASS = {05E10 (05E05 33D80 81R10)},
  MRNUMBER = {1434225},
MRREVIEWER = {Thomas\ Scharf},
       DOI = {10.1063/1.531807},
       URL = {https://doi.org/10.1063/1.531807},
}

@incollection{LT,
    AUTHOR = {Leclerc, Bernard and Thibon, Jean-Yves},
     TITLE = {Littlewood-{R}ichardson coefficients and {K}azhdan-{L}usztig
              polynomials},
 BOOKTITLE = {Combinatorial methods in representation theory ({K}yoto,
              1998)},
    SERIES = {Adv. Stud. Pure Math.},
    VOLUME = {28},
     PAGES = {155--220},
 PUBLISHER = {Kinokuniya, Tokyo},
      YEAR = {2000},
      ISBN = {4-314-10141-5},
   MRCLASS = {20C08 (33D80)},
  MRNUMBER = {1864481},
MRREVIEWER = {Andrew\ Mathas},
       DOI = {10.2969/aspm/02810155},
       URL = {https://doi.org/10.2969/aspm/02810155},
}

@article{KT,
    AUTHOR = {Kashiwara, Masaki and Tanisaki, Toshiyuki},
     TITLE = {Parabolic {K}azhdan-{L}usztig polynomials and {S}chubert
              varieties},
   JOURNAL = {J. Algebra},
  FJOURNAL = {Journal of Algebra},
    VOLUME = {249},
      YEAR = {2002},
    NUMBER = {2},
     PAGES = {306--325},
      ISSN = {0021-8693,1090-266X},
   MRCLASS = {14M15 (14F43)},
  MRNUMBER = {1901161},
       DOI = {10.1006/jabr.2000.8690},
       URL = {https://doi.org/10.1006/jabr.2000.8690},
}

@article{Mac88,
    AUTHOR = {Macdonald, I. G.},
     TITLE = {A new class of symmetric functions},
   JOURNAL = {S\'em. Lothar. Combin.},
  FJOURNAL = {S\'eminaire Lotharingien de Combinatoire},
    VOLUME = {20},
      YEAR = {1988},
     PAGES = {Art. B20a, 41 pp.},
      NOTE = {Publ. I.R.M.A. Strasbourg, 372/S--20},
}

@book{Mac,
    AUTHOR = {Macdonald, I. G.},
     TITLE = {Symmetric functions and {H}all polynomials},
    SERIES = {Oxford Mathematical Monographs},
   EDITION = {Second},
      NOTE = {With contributions by A. Zelevinsky,
              Oxford Science Publications},
 PUBLISHER = {The Clarendon Press, Oxford University Press, New York},
      YEAR = {1995},
     PAGES = {x+475},
      ISBN = {0-19-853489-2},
   MRCLASS = {05E05 (05-02 20C30 20C33 20K01 33C80 33D80)},
  MRNUMBER = {1354144},
MRREVIEWER = {John\ R.\ Stembridge},
}

@article{HRW,
    AUTHOR = {Haglund, J. and Remmel, J. B. and Wilson, A. T.},
     TITLE = {The delta conjecture},
   JOURNAL = {Trans. Amer. Math. Soc.},
  FJOURNAL = {Transactions of the American Mathematical Society},
    VOLUME = {370},
      YEAR = {2018},
    NUMBER = {6},
     PAGES = {4029--4057},
      ISSN = {0002-9947,1088-6850},
   MRCLASS = {05E05},
  MRNUMBER = {3811519},
MRREVIEWER = {David\ J.\ Grabiner},
       DOI = {10.1090/tran/7096},
       URL = {https://doi.org/10.1090/tran/7096},
}

@article{DM,
    AUTHOR = {D'Adderio, Michele and Mellit, Anton},
     TITLE = {A proof of the compositional delta conjecture},
   JOURNAL = {Adv. Math.},
  FJOURNAL = {Advances in Mathematics},
    VOLUME = {402},
      YEAR = {2022},
     PAGES = {Paper No. 108342, 17},
      ISSN = {0001-8708,1090-2082},
   MRCLASS = {05E05},
  MRNUMBER = {4401822},
MRREVIEWER = {Sam\ Hopkins},
       DOI = {10.1016/j.aim.2022.108342},
       URL = {https://doi.org/10.1016/j.aim.2022.108342},
}

@article{HaimanDH,
    AUTHOR = {Haiman, Mark},
     TITLE = {Vanishing theorems and character formulas for the {H}ilbert
              scheme of points in the plane},
   JOURNAL = {Invent. Math.},
  FJOURNAL = {Inventiones Mathematicae},
    VOLUME = {149},
      YEAR = {2002},
    NUMBER = {2},
     PAGES = {371--407},
      ISSN = {0020-9910,1432-1297},
   MRCLASS = {14C05 (05E05 14F17 14R20)},
  MRNUMBER = {1918676},
MRREVIEWER = {Claudio\ Procesi},
       DOI = {10.1007/s002220200219},
       URL = {https://doi.org/10.1007/s002220200219},
}

@article{HMZ,
    AUTHOR = {Haglund, J. and Morse, J. and Zabrocki, M.},
     TITLE = {A compositional shuffle conjecture specifying touch points of
              the {D}yck path},
   JOURNAL = {Canad. J. Math.},
  FJOURNAL = {Canadian Journal of Mathematics. Journal Canadien de
              Math\'{e}matiques},
    VOLUME = {64},
      YEAR = {2012},
    NUMBER = {4},
     PAGES = {822--844},
      ISSN = {0008-414X,1496-4279},
   MRCLASS = {05E05 (33D52)},
  MRNUMBER = {2957232},
MRREVIEWER = {Arthur\ L. B. Yang},
       DOI = {10.4153/CJM-2011-078-4},
       URL = {https://doi.org/10.4153/CJM-2011-078-4},
}

@article{BHMPS,
    AUTHOR = {Blasiak, Jonah and Haiman, Mark and Morse, Jennifer and Pun,
              Anna and Seelinger, George H.},
     TITLE = {A proof of the extended delta conjecture},
   JOURNAL = {Forum Math. Pi},
  FJOURNAL = {Forum of Mathematics. Pi},
    VOLUME = {11},
      YEAR = {2023},
     PAGES = {Paper No. e6, 28},
      ISSN = {2050-5086},
   MRCLASS = {05A19 (20G05 35P15)},
  MRNUMBER = {4553915},
MRREVIEWER = {Eric\ S.\ Egge},
       DOI = {10.1017/fmp.2023.3},
       URL = {https://doi.org/10.1017/fmp.2023.3},
}

@article{Hai01,
    AUTHOR = {Haiman, Mark},
     TITLE = {{H}ilbert schemes, polygraphs and the {M}acdonald positivity
              conjecture},
   JOURNAL = {J. Amer. Math. Soc.},
  FJOURNAL = {Journal of the American Mathematical Society},
    VOLUME = {14},
      YEAR = {2001},
    NUMBER = {4},
     PAGES = {941--1006},
}

@article{DF17,
    AUTHOR = {Do{\l}{\polhk e}ga, Maciej and F\'{e}ray, Valentin},
     TITLE = {Cumulants of {J}ack symmetric functions and the
              {$b$}-conjecture},
   JOURNAL = {Trans. Amer. Math. Soc.},
  FJOURNAL = {Transactions of the American Mathematical Society},
    VOLUME = {369},
      YEAR = {2017},
    NUMBER = {12},
     PAGES = {9015--9039},
}

@article{GJ96,
    AUTHOR = {Goulden, I. P. and Jackson, D. M.},
     TITLE = {Connection coefficients, matchings, maps and combinatorial
              conjectures for {J}ack symmetric functions},
   JOURNAL = {Trans. Amer. Math. Soc.},
  FJOURNAL = {Transactions of the American Mathematical Society},
    VOLUME = {348},
      YEAR = {1996},
    NUMBER = {3},
     PAGES = {873--892},
}

@article{HNY20,
    AUTHOR = {Huh, JiSun and Nam, Sun-Young and Yoo, Meesue},
     TITLE = {Melting lollipop chromatic quasisymmetric functions and {S}chur
              expansion of unicellular {LLT} polynomials},
   JOURNAL = {Discrete Math.},
    VOLUME = {343},
      YEAR = {2020},
    NUMBER = {3},
     PAGES = {111728},
       DOI = {10.1016/j.disc.2019.111728},
}

\end{document}